\documentclass[12pt]{article}
\usepackage{graphicx} % Required for inserting images
\usepackage{amsmath}
\usepackage{amsthm}
\usepackage{mathrsfs}
\usepackage{amssymb}
\usepackage{tikz} 
\usetikzlibrary{positioning,shapes.geometric}
\usepackage[margin=1in]{geometry}
\usepackage[inline]{enumitem}
\usepackage{appendix}
\usepackage{soul}

\usepackage[colorlinks=true, linkcolor=blue, citecolor=blue,
pagebackref=true]{hyperref}

\title{Limitations of deducing measures of limsup sets from measures of finite intersections}
\author{Charlie Wilson}
\date{}

\begin{document}
\setstcolor{orange}
\maketitle

\newtheorem{lemma}{Lemma}[section]
\newtheorem*{definition}{Definition}
\newtheorem*{remark}{Remark}
\newtheorem{theorem}{Theorem}[section]
\newtheorem*{theorem*}{Theorem}
\newtheorem{corollary}{Corollary}[section]
\newtheorem{prop}{Proposition}[section]
\newtheorem{question}{Question}

\newtheorem*{theoremKS*}{Kochen--Stone Theorem (1964)}
\newtheorem*{theoremSV*}{Theorem BV1}
\newtheorem*{theoremBCC*}{Convergence Borel--Cantelli Lemma (1909)}
\newtheorem*{theoremBCD*}{Divergence Borel--Cantelli Lemma (1909)}
\newtheorem*{theoremFrolov*}{Theorem Frolov 2012}

\begin{abstract}
   Early results by Borel and Cantelli and Chung (Rendiconti del Circolo Matematico di Palermo, 1909) and Erd\H{o}s (Transactions of the AMS, 1952) have provided bounds for the measure of a limsup set in terms of measures of its constituent sets and their intersections. Recent work by Beresnevich and Velani (Journal of Mathematical Analysis and its Applications, 2023) states that for sequences of balls the measure of the corresponding limsup set being positive is equivalent to a condition on the relationship between the measures of these balls and their pairwise intersections. In this paper, we show that the condition that the sets are balls is strictly necessary in the result of Beresnevich and Velani. Moreover, let $d \in \mathbb{N}$ and let $[0,1]^d$ be equipped with Lebesgue measure $\mu$. Fix $m \in \mathbb{N}$. When we drop the condition that the sets are balls, we can find two sequences of sets $(A_i)_{i \in \mathbb{N}}$ and $(B_i)_{i \in \mathbb{N}}$ in $[0,1]^d$ such that $\mu(A_i)=\mu(B_i)$ for all $i \in \mathbb{N}$ and for any sequence $(i_1,i_2,...,i_l)$ where for all $l \leq m$ we have $\mu(A_{i_1}\cap A_{i_2} \cap... \cap A_{i_l})=\mu(B_{i_1}\cap B_{i_2} \cap... \cap B_{i_l})$ but $\mu(\limsup_{i \rightarrow \infty} A_i)=1$ and $\mu(\limsup_{i \rightarrow \infty} B_i)=0$.
  
\end{abstract}

\subsubsection*{Keywords}
Measure theory, Borel--Cantelli Lemma, Diophantine approximation, Probability theory

AMS subject classification
28A75 
11K60,
37A50.

\subsubsection*{Acknowledgements} 
Thanks to Mark Holland and Demi Allen for the ongoing support and help with the project. Also, thanks to Felipe Ram\'{\i}rez for the advice and suggestions. Thanks also to the backing of DWN, CBR and HBY. This research was funded by the University of Exeter. For the purpose of open access, the author has applied a Creative Commons Attribution (CC BY) licence to any Author Accepted Manuscript version arising from this submission.

\section{Introduction}
Let $(\Omega, \mathscr{A}, \mu)$ be a probability space and let $(A_i)_{i \in \mathbb{N}}$ be a sequence of $\mu$-measurable sets in~$\Omega$.  The \textit{limsup} of the sequence $(A_i)_{i \in \mathbb{N}}$ is defined as:

\[ \limsup_{i \rightarrow \infty} A_i= \bigcap_{n \geq 1}\bigcup_{i \geq n} A_i= \{x \in \Omega : x \in A_i \text{ for infinitely many } i \in \mathbb{N} \}  ~.\]

In areas such as metric number theory, Diophantine approximation, and shrinking target problems in dynamical systems, we often wish to determine the measures of limsup sets. In many situations, the Borel--Cantelli Lemmas can help to do just this (see statements of the lemmas below). Statements and proofs of the Borel--Cantelli Lemmas can be found in many probability textbooks, see, for example, \cite{billingsley1995probability}. 

\begin{theoremBCC*}\label{BCconv}
Let $(\Omega, \mathscr{A}, \mu)$ be a probability space and let $(A_i)_{i \in \mathbb{N}}$ be a sequence of $\mu$-measurable sets in $\Omega$. Suppose that $\sum_{i=1}^{\infty} \mu(A_i)<\infty$, then $\mu(\limsup_{i \rightarrow \infty} A_i)=0$.
\end{theoremBCC*}

\begin{theoremBCD*}\label{BCdiv}
Let $(\Omega, \mathscr{A}, \mu)$ be a probability space and let $(A_i)_{i \in \mathbb{N}}$ be a sequence of $\mu$-measurable sets in $\Omega$. Suppose further that $A_i$ and $A_j$ are pairwise independent, i.e. $\mu(A_i \cap A_j)=\mu(A_i)\mu(A_j)$ for all $i,j \in \mathbb{N}$ with $i \neq j$. If $\sum_{i=1}^{\infty} \mu(A_i)=\infty$, then we have $\mu(\limsup_{i \rightarrow \infty} A_i)=1$.
\end{theoremBCD*}

This latter result was further extended by Chung and Erd\H{o}s \cite{ErdosChungpaper}, Erd\H{o}s and Renyi \cite{ErdosRenyi}, and Kochen and Stone \cite{kochenstone} to yield the following:

\begin{theoremKS*}\label{Kochen stone theorem}
Let $(\Omega, \mathscr{A}, \mu)$ be a probability space and let $(A_i)_{i \in \mathbb{N}}$ be a sequence of $\mu$-measurable sets in $\Omega$. Suppose that $\sum_{i=1}^{\infty} \mu(A_i)=\infty$. Then
    \begin{equation} \label{quasi independence}
        \mu\left(\limsup_{i \rightarrow \infty} A_i \right) \geq \limsup_{n \rightarrow \infty} \frac{\left( \sum_{s=1}^n \mu(A_s) \right)^2}{\sum_{s,t=1}^n \mu(A_s \cap A_t)} ~.
    \end{equation} 
\end{theoremKS*}

The latter result is indeed a strengthening of the Divergence Borel--Cantelli Lemma, as the expression on the right hand side of~(\ref{quasi independence}) will equal $1$ if the $A_i$ are pairwise independent. Despite this result being an improvement, there are examples where the measure of the limsup of $A_i$ is 1 but the right-hand side of the inequality is 0 and so the bound can be ineffective. Some other approaches have been taken recently to find new bounds for finding the measure of the limsup set. For example, Feng, Li and Shen provided a weighted version of the Kochen--Stone Theorem in \cite{weightedBCpaper}. Petrov \cite{Petrov04} took a variational principle approach to finding a bound. Readers interested in a full history of these results are referred to \cite{Frolov2012}.

In the case where the constituent sets of the sequence are balls, recent work by Beresnevich and Velani \cite{Velanipaper} provides necessary and sufficient conditions for a limsup set to have strictly positive measure. We begin with some notation and definitions before stating their result. For a Borel probability measure $\mu$ on the space $\Omega$, we let supp($\mu$) denote the support of the measure $\mu$. Furthermore given $x \in \Omega$ and $r > 0$, we denote the ball centred at $x$ of radius $r$ to be $B(x,r)$. Given a real number $a > 0$, we denote by $aB$ the ball $B$ scaled by a factor $a$, that is
$aB := B(x, ar)$. A measure $\mu$ is said to be \textit{doubling} if there are constants $\lambda \geq 1$ and $r_0 >0$ such that for any $x \in \text{supp}(\mu)$ and $0 <r<r_0$ we have
\[\mu(B(x,2r)) \leq \lambda \mu(B(x,r)).\]

The following result is Theorem 3 from \cite{Velanipaper}, which provides a very useful technique for deducing that a limsup set of balls has positive measure.

\begin{theoremSV*}
\label{3rd Velani result}
     Let $(\Omega, \mathscr{A}, \mu, d)$ be a metric measure space equipped with a doubling Borel probability measure $\mu$. Let $ (B_i)_{i \in \mathbb{N}} $ be a sequence of balls in $\Omega$ such that there exist constants $a,b>1 $ such that $\mu(aB_i) \leq b \mu(B_i) $ holds for all $i$ sufficiently large. Then,
     \[\mu\left(\limsup_{i \rightarrow \infty} B_i \right)>0\]
     if and only if there exists a subsequence $( B_{i_n} )_{n\in \mathbb{N}}$ of $( B_i )_{i\in \mathbb{N}}$ and a constant $C>0$ such that
      \[\sum_{n=1}^{\infty} \mu(B_{i_n})=\infty \]
      and for infinitely many $Q \in \mathbb{N}$ we have 
     \[ \sum_{s,t=1}^{Q} \mu(B_{i_s} \cap B_{i_t}) \leq C \left(\sum_{s=1}^{Q} \mu(B_{i_s})\right)^2 ~.\]
    
\end{theoremSV*}

     Moreover, we note that in some cases it is sufficient to show that the measure of the limsup of the sequence of sets is positive to show that the limsup set has full measure. This is because there are several `zero-one' laws that state that provided the sequence of sets has additional structure (for example the sets can be seen as a tail event), the measure of the limsup can only take the values $0$ or $1$. See, for example, \cite{billingsley1995probability} for Kolmogorov's zero-one law, \cite{cassels} for Cassels zero-one law, and \cite{Gallagher} for Gallaghers zero-one law. 
    
    Another way of reformulating Theorem BV1 is as follows

    \begin{corollary}\label{the corollary}
        Let $(\Omega, \mathscr{A}, \mu, d)$ be a metric measure space equipped with a doubling Borel probability measure $\mu$. Let $(A_i)_{i \in \mathbb{N}}$ and $ (B_i)_{i \in \mathbb{N}} $ be sequences of balls in $\Omega$ such that there exist constants $a_1,a_2,b_1,b_2>1 $ such that $\mu(a_1B_i) \leq b_1 \mu(B_i) $ and $\mu(a_2B_i) \leq b_2 \mu(B_i) $  hold for all $i$ sufficiently large.
        Suppose that
        \begin{equation}\label{measure equal in cor}
            \begin{split}
    \mu (A_i) &= \mu (B_i),  \quad  \text{for all}~ i \in \mathbb{N}, \text{ and} \\
     \mu (A_i \cap A_j) &= \mu (B_i \cap B_j),  \quad \text{for all} ~ i, j \in \mathbb{N}.
    \end{split}
        \end{equation}
        Then,
        \begin{align*}
            \mu \left( \limsup_{i\rightarrow \infty} B_i \right) >0 \implies \mu \left(\limsup_{i \rightarrow \infty}A_i \right)>0.
        \end{align*}
    \end{corollary}

    This raises the question of whether the condition that the sets in the sequences $(A_i)_{i \in \mathbb{N}}$ and $(B_i)_{i \in \mathbb{N}}$ are balls is strictly necessary for the conclusion to hold.

\begin{question} \label{Q}
Consider two sequences of sets, $(A_i)_{i \in \mathbb{N}}$ and $(B_i)_{i \in \mathbb{N}}$ in $\Omega$, whose measures are such that
\begin{equation}
\label{the very first intersection requirement}
\begin{split}
    \mu (A_i) &= \mu (B_i),  \quad  \text{for all}~ i \in \mathbb{N}, \text{ and} \\
     \mu (A_i \cap A_j) &= \mu (B_i \cap B_j),  \quad \text{for all} ~ i, j \in \mathbb{N}.
    \end{split}
\end{equation}
Is it true that $\mu(\limsup_{i\rightarrow \infty} B_i)>0  \implies  \mu(\limsup_{i \rightarrow \infty} A_i)>0$ ?
\end{question}

We give examples answering Question \ref{Q} in the negative and investigate the consequences if we ask for control of measures of some of the higher intersections. We note that the requirements in Theorem BV1 of the space being doubling and the enlargement condition on the sequences are fairly mild, and so we neglect them in the question for brevity. Indeed, the examples rendering the answer to the question negative satisfy these requirements. Our first main result is the following, which answers Question \ref{Q} for the case $m=2$ and also gives answers for the question fixing the measures of higher intersections.

\begin{theorem} \label{the equality theorem}
    Fix $m\in \mathbb{N}$ and $c \in [0,1]$ and let $\mu$ denote Lebesgue measure. Then there exist sequences $(A_i)_{i \in \mathbb{N}}$ and $(B_i)_{i \in \mathbb{N}}$ of Borel subsets of $[0,1]$ satisfying:
\begin{enumerate}[label=(\roman*)]
    \item $\mu (A_i)  = \mu (B_i)$  for all $i \in \mathbb{N}$, 
     \item For any sequence of natural numbers $i_1,i_2,...,i_l$ with $1 \leq l \leq m$, we have 
     \[\mu(A_{i_1} \cap ... \cap A_{i_l})  = \mu(B_{i_1} \cap ... \cap B_{i_l}),\]
     \item $\mu(\limsup_{i \rightarrow \infty} A_i) =0$, and
     \item $\mu(\limsup_{i \rightarrow \infty} B_i) =c$.
    \end{enumerate}

\end{theorem}

    Going further, the answer to Question $1$ is still `no' even if we replace the conditions of the equality of measures of sets and measures of intersections with relevant strict inequalities. We see the motivating Corollary \ref{the corollary} would remain true in the case we replaced \eqref{measure equal in cor} with $\mu(A_i)>\mu(B_i) $ for all $i \in \mathbb{N}$ and $\mu(A_i \cap A_j) <\mu(B_i \cap B_j)$ for all distinct $i,j \in \mathbb{N}$. This is indeed a problem of a distinct nature, and the examples used to answer Question 1 in the case of \eqref{the very first intersection requirement} containing equalities or inequalities are very different.

As in the case for equalities, we go further and consider the problem in which we control the measures of several intersections of distinct sets. Unlike in the equality case, we are required to choose the appropriate direction for the inequalities. We can observe a heuristic for this by considering the inclusion exclusion principle and rewriting the measure of the $\limsup$ set as the limit of the measures of sets. To this end, by the continuity of measures, we may express the measure of the limsup of a sequence of sets as
\begin{align}
  \mu\left(\limsup_{i \rightarrow \infty} A_i \right) = \lim_{n\rightarrow \infty} \lim_{m\rightarrow \infty} \mu \left(\bigcup_{i = n}^m A_i \right) \label{limit union} ~.
\end{align}
 Therefore, to control the measure of the limsup set it is sufficient to control the measure of finite unions of the constituent sets. By the inclusion exclusion principle (see, for example, Chapter 2 of \cite{billingsley1995probability}) we see that
\begin{align}\label{inclusionexclusion}
    \mu \left(\bigcup_{i=n}^{m} A_i \right) = \sum_{k=1}^{m-n+1} (-1)^{k+1} \sum_{n \leq i_1 < \cdots < i_k \leq m} \mu ( A_{i_1} \cap \cdots \cap A_{i_k} ) ~.
\end{align}
        The parity of the coefficients on the right-hand side will be used to determine which quantities we would heuristically want to make smaller and larger to increase the measure of the limsup of the sequence $(A_i)_{i \in \mathbb{N}}$. If the number of distinct sets being intersected is even, then the coefficient is negative, and we want these to be small, so that the measure of the limsup is larger. Similarly, if the number of distinct sets in the intersection is odd, then the coefficient is positive, and we would want these to be large to increase the measure of the limsup. This justifies our choice of inequalities in $(ii)$ of Theorem \ref{the big theorem} stated next. We remark this theory is consistent with the Kochen--Stone Theorem and results by Frolov \cite{Frolov2012}.

\begin{theorem} \label{the big theorem}
 Fix $m\in \mathbb{N}$ and $c \in [0,1]$ and let $\mu$ denote Lebesgue measure. Then there exist sequences $(A_i)_{i \in \mathbb{N}}$ and $(B_i)_{i \in \mathbb{N}}$ of Borel subsets of $[0,1]$ satisfying:
    \begin{enumerate}[label=(\roman*)]
    \item $\mu (A_i)  > \mu (B_i)$  for all $i \in \mathbb{N}$, 
     \item For any sequence of natural numbers $i_1,i_2,...,i_l$ with $1 \leq l \leq m$, we have \[(-1)^l\mu(A_{i_1} \cap ... \cap A_{i_l}) < (-1)^l \mu(B_{i_1} \cap ... \cap B_{i_l}),\] 
     \item $\mu(\limsup_{i \rightarrow \infty} A_i) =0$, and
     \item $\mu(\limsup_{i \rightarrow \infty} B_i) =c$.
    \end{enumerate}

\end{theorem}

We see from both Theorem \ref{the big theorem} and Theorem \ref{the equality theorem} that controlling the measures of arbitrary $l$-wise intersections between sets (for $1 \leq l \leq m $) does not necessarily fix the measure of the limsup set. The work of Frolov \cite{Frolov2012} and references therein give improvements on the Borel--Cantelli lemmas by controlling the measure of triple and quadruple intersections.  One of the consequences of Theorem $\ref{the equality theorem}$ is that for any natural number $m$ we cannot hope to form a precise formula for the measure of the limsup of a sequence of sets, dependent only on the measures of the $l$-wise intersections of the sets for $l \leq m$. In our proof of Theorem \ref{the equality theorem} we give examples in which we have two sequences with identical $l$ wise intersection measures for $l \leq m$, yet the measure of the limsup set of the two sequences can be vastly different. Hence, there is a fundamental geometric aspect of limsup sets for which these finite intersection measures are too crude to recognise.

For completeness, we observe that if, on the other hand, $(A_i)_{i \in \mathbb{N}}$ and $(B_i)_{i \in \mathbb{N}}$ are sequences of sets in $(\Omega, \mathcal{A},\mu)$ and we insist that $\mu (A_{i_1} \cap A_{i_2} \cap ... \cap A_{i_l})=\mu (B_{i_1} \cap B_{i_2} \cap ... \cap B_{i_l})$ for \textbf{every} length sequence $l \in \mathbb{N}$, then $\mu(\limsup_{i \rightarrow\infty} A_i)=\mu(\limsup_{i \rightarrow \infty} B_i)$.

\begin{theorem} \label{fix everything}
     Let $(\Omega_1, \mathscr{A}_1, \mu_1)$ and  $(\Omega_2, \mathscr{A}_2, \mu_2)$ be two (potentially distinct) probability spaces. Let $(A_i)_{i\in \mathbb{N}}$ and $(B_i)_{i\in \mathbb{N}}$ be sequences of measurable sets in $\Omega_1$ and $\Omega_2$ respectively. Suppose $\mu_1(\limsup_{i \rightarrow \infty} A_i)=c$ with $0 \leq c \leq 1$. 
     Furthermore, suppose that for any $k \in \mathbb{N}$ and indices $i_1,i_2,...,i_k \in \mathbb{N}$ we have $\mu_1 (A_{i_1} \cap A_{i_2} \cap ... \cap A_{i_k})=\mu_2 (B_{i_1} \cap B_{i_2} \cap ... \cap B_{i_k})$.
     Then, $\mu_2 (\limsup_{i \rightarrow \infty} B_i)=c$.
\end{theorem}

\begin{proof}
We first note that by the inclusion-exclusion principle as stated earlier, since $\mu_1$ and $\mu_2$ are finite measures, for any $n \in \mathbb{N}$ we have:
    \begin{equation}\label{inclusion exclusion}
        \mu_1 \left(\bigcup_{i=1}^{n} A_i \right) = \sum_{k=1}^{n} (-1)^{k+1} \sum_{1 \leq i_1 < \cdots < i_k \leq n} \mu_1 ( A_{i_1} \cap \cdots \cap A_{i_k} ) ~\text{and}
    \end{equation}
    \begin{equation}\label{inclusion exclusion for B}
        \mu_2 \left(\bigcup_{i=1}^{n} B_i \right) = \sum_{k=1}^{n} (-1)^{k+1} \sum_{1 \leq i_1 < \cdots < i_k \leq n} \mu_2 ( B_{i_1} \cap \cdots \cap B_{i_k} ) ~.
    \end{equation}
Hence, for all $k \in \mathbb{N}$ and $n\in \mathbb{N}$ we have:
\begin{equation}\label{inclusion}
    \mu_1 \left( \bigcup_{i=k}^{n} A_i \right) = \mu_2 \left( \bigcup_{i=k}^{n} B_i \right) ~.
\end{equation}
For $r \geq k$, define \[ \Tilde{A}_{k,r}= \bigcup_{i=k}^{r} A_i, \text{ and } \Tilde{B}_{k,r}= \bigcup_{i=k}^{r} B_i ~. \]
It follows from (\ref{inclusion}) that, for any $r,k \in \mathbb{N}$ with $r \geq k$, we have
\[\mu_1(\Tilde{A}_{k,r})=\mu_2(\Tilde{B}_{k,r}) ~.\]
Since
\[ \Tilde{A}_{k,r} \subset  \Tilde{A} _{k,r+1}\subset \Tilde{A}_{k,r+2} \subset \Tilde{A}_{k,r+3} \subset ... \text{ and } \Tilde{B}_{k,r} \subset  \Tilde{B} _{k,r+1}\subset \Tilde{B}_{k,r+2} \subset \Tilde{B}_{k,r+3} \subset ... ~,\]
it follows by continuity of measure that
\[ \mu_1 \left( \bigcup_{i=k}^{\infty} A_i \right) = \lim_{r \rightarrow \infty} \mu_1(\Tilde{A}_{k,r}) =\lim_{r \rightarrow \infty} \mu_2(\Tilde{B}_{k,r})= \mu_2 \left(\bigcup_{i=k}^{\infty} B_i \right) ~.\]
Furthermore, by a set containment argument, we have that:
\begin{align*}
  \mu_1\left(\limsup_{i \rightarrow \infty} A_i \right) &= \lim_{k\rightarrow \infty} \lim_{r\rightarrow \infty} \mu_1 \left(\bigcup_{i = k}^r A_i \right), \text{ and} \\
  \mu_2 \left(\limsup_{i \rightarrow \infty} B_i \right) &= \lim_{k\rightarrow \infty} \lim_{r\rightarrow \infty} \mu_2 \left(\bigcup_{i = k}^r B_i \right).
\end{align*}
Thus, the claimed result follows. \qedhere

\end{proof}

We can easily adapt this proof to provide a statement in which, instead of having equality signs between the measures of the intersections, we instead have inequalities. 

\begin{theorem} \label{fix everything inequality}
     Let $(\Omega_1, \mathscr{A}_1, \mu_1)$ and  $(\Omega_2, \mathscr{A}_2, \mu_2)$ be two (potentially distinct) probability spaces. Let $(A_i)_{i\in \mathbb{N}}$ be a sequence of measurable sets in $\Omega_1$ and $(B_i)_{i\in \mathbb{N}}$ be a sequence in $\Omega_2$. Suppose for any $k \in \mathbb{N}$ and indices $i_1<i_2<...<i_k$ that
     \begin{align*}
         \mu_1 (A_{i_1} \cap A_{i_2} \cap ... \cap A_{i_k}) &\geq \mu_2 (B_{i_1} \cap B_{i_2} \cap ... \cap B_{i_k}) \text{ for k odd and} \\
         \mu_1 (A_{i_1} \cap A_{i_2} \cap ... \cap A_{i_k}) &\leq \mu_2 (B_{i_1} \cap B_{i_2} \cap ... \cap B_{i_k}) \text{ for k even.}
     \end{align*}
     Then these conditions imply $\mu_1(\limsup_{i \rightarrow \infty} A_i) \geq \mu_2 (\limsup_{i \rightarrow \infty} B_i)$.
\end{theorem}

\begin{proof}
    The argument is very similar to the proof of Theorem 1.3 (so we omit it here) but with equality signs now replaced with the appropriate inequalities. Explicitly from (\ref{inclusion exclusion}) and (\ref{inclusion exclusion for B}) the odd intersection terms have positive coefficients and the even intersection terms have negative coefficients.
\end{proof}

The rest of the paper is spent proving the two main results stated in the Introduction. In Section \ref{equality proof section} we prove Theorem \ref{the equality theorem} and in Section \ref{the k intersection case} we prove Theorem \ref{the big theorem}.

\section{Proof of Theorem \ref{the equality theorem}}\label{equality proof section}
The proof of the theorem is divided into two parts. The first of which is a lemma that establishes the existence of two finite sequences of sets with the measure intersection properties as in the hypothesis of the theorem. Each set is composed of a union of dyadic intervals. We also see that the union of one of the sequences of sets covers the entirety of the unit interval, and the union of sets in the other sequence avoids a subinterval of $[0,1]$. The second part of this section takes these sequences and, by shrinking and intersecting them, builds the two infinite sequences of sets in the statement of the theorem.

Before proceeding with the proof of the theorem, we begin with an auxiliary lemma to define sets that will appear throughout the proof. Here and throughout for $m \in \mathbb{N}$ we let $\mathcal{D}_m = \left\{\left[\frac{j}{2^m}, \frac{j+1}{2^m}\right): j \in \mathbb{Z
}, ~ 0 \leq j \leq 2^m-1 \right\}$ denote the set of dyadic intervals of length $2^{-m}$ in $[0,1]$. 
\begin{lemma}\label{prelim equality lemma}
    For any $m \in \mathbb{N}$, there exist two finite collections  $C_1,\allowbreak C_2,\allowbreak \dots,\allowbreak C_{m+1}$ and
$D_1,\allowbreak D_2,\allowbreak \dots,\allowbreak D_{m+1}$ of subsets of $[0,1]$ such that the following hold: 
    \begin{enumerate}[label=(\roman*)]
    \item $\mu (C_i) = \mu (D_i) $ for all $1\leq i \leq m+1$ ,
     \item For any sequence of indices $1 \leq i_1,i_2,...,i_l \leq m+1$ with $1 \leq l \leq m$, we have \[\mu(C_{i_1} \cap ... \cap C_{i_l}) = \mu(D_{i_1} \cap ... \cap D_{i_l}),\]

     \item $1=\mu \left( \bigcup_{i=1}^{m+1} D_i \right) = \mu \left( \bigcup_{i=1}^{m+1} C_i \right)+\frac{1}{2^m}, and$     
     \item  for each $1 \leq i \leq m+1$, $C_i$ and $D_i$ are unions of dyadic intervals of length $2^{-m}$.
     \end{enumerate}

\end{lemma}

\begin{remark}
    We will construct each $D_i$ to be a union of intervals from $\mathcal{D}_m$ such that each interval from $\mathcal{D}_m$ appears in at least one of the $D_i$. We will then perform a similar procedure for the $C_i$ while ensuring that the measures of the sets and intersections satisfy the conditions of the lemma. To ensure that condition $(iii)$ is met, there will be one interval from $\mathcal{D}_m$ that will not appear in any of the $C_i$. We will throughout refer to the size or cardinality of a set to be the number of distinct elements it contains. Note that the empty set is a set of size 0.
\end{remark}

\begin{proof}[Proof of Lemma \ref{prelim equality lemma}]
    We start by considering the collection $\{D_1,D_2,...,D_{m+1}\}$.
    We list all subsets of \mbox{$\{D_1,\allowbreak D_2,\allowbreak \dots,\allowbreak D_{m+1}\}$} consisting of an odd number of $D_i$ and call this collection $D^{\textrm{odd}}$. Similarly, we define $D^{\textrm{even}}$ to be the collection of subsets of $\{D_1,D_2,\dots,D_{m+1}\}$ consisting of an even number of elements. For each of the subsets in $D^{\textrm{odd}}$, we assign a distinct (not assigned to any of the other subsets) interval of size $2^{-m}$ from $\mathcal{D}_m$ and then add this interval to all the $D_i$ in this subset. We note that there are $2^m$ nonempty subsets of odd size and also $2^m$ dyadic intervals; therefore, this procedure is well defined and ensures that each interval is contained in at least one of the $D_i$. This ensures that:
    \[\mu \left( \bigcup_{i=1}^{m+1}D_i \right) =1.\]
    To aid the reader, for the case $m=3$ we include an example of the correspondence between the collection $D^{\textrm{odd}}$ to the dyadic intervals $\mathcal{D}_3$ in Table $\ref{correspondence table}$. Furthermore, we include the analogous examples of the sets $D_i$ in Figure \ref{Di set example m=3}.

\begin{table}[h]
        \centering \caption{An example of the correspondence between  between $D^{\textrm{odd}}$ and the dyadic intervals $\mathcal{D}_m$ in the case $m=3$. }
\vspace{10pt}
{\renewcommand{\arraystretch}{2}
    \begin{tabular}{|c|c|c|c|c|}
\hline
Collection of $D_i$ & $\{D_1 \}$ & $\{D_2\}$ & $\{D_1,D_3,D_4\}$  & $\{D_2,D_3,D_4\}$  \\
\hline
Associated interval & $ \left[ \frac{3}{8}, \frac{4}{8} \right] $ & $ \left[ \frac{5}{8}, \frac{6}{8} \right] $ & $ \left[ \frac{2}{8}, \frac{3}{8} \right] $ & $ \left[ \frac{4}{8}, \frac{5}{8} \right] $ \\
\hline
\hline
Collection of $D_i$ & $\{D_3 \}$ & $\{D_4\}$ & $\{D_1,D_2,D_3\}$  & $\{D_1,D_2,D_4\}$  \\
\hline
Associated interval & $ \left[ \frac{6}{8}, \frac{7}{8} \right] $ & $ \left[ \frac{7}{8}, \frac{8}{8} \right] $ & $ \left[ \frac{0}{8}, \frac{1}{8} \right] $ & $ \left[ \frac{1}{8}, \frac{2}{8} \right] $ \\
\hline
    \end{tabular}}
    \label{correspondence table}
\end{table}

\begin{figure}[!ht]
    \caption{An example of the sets $D_i$ in the case $m=3$.}
    \label{leb construction Di}
\begin{center}
    \begin{tikzpicture}
    \draw[very thick] (0,5) -- (12,5);
    \draw (0,4) -- (12/8*4,4);
    \draw (0,3) --(2*12/8,3);
    \draw (4*12/8,3) --(6*12/8,3);
    \draw (0,2) -- (12/8,2);
    \draw (2*12/8,2) -- (3*12/8,2);
    \draw (4*12/8,2) -- (5*12/8,2);
    \draw (6*12/8,2) -- (7*12/8,2);
    \draw (12/8,1) -- (2*12/8,1);
    \draw (2*12/8,1) -- (3*12/8,1);
    \draw (4*12/8,1) -- (5*12/8,1);
    \draw (7*12/8,1) -- (8*12/8,1);
    \node at (-1/2,5) {[0,1]};
    \node at (-1/2,4) {$D_1$};
    \node at (-1/2,3) {$D_2$};
    \node at (-1/2,2) {$D_3$};
    \node at (-1/2,1) {$D_4$};
    \node at (0,1/2) {0};
    \node at (12/8,1/2) {$\frac{1}{8}$};
    \node at (24/8,1/2) {$\frac{2}{8}$};
    \node at (3*12/8,1/2) {$\frac{3}{8}$};
    \node at (4*12/8,1/2) {$\frac{4}{8}$};
    \node at (5*12/8,1/2) {$\frac{5}{8}$};
    \node at (6*12/8,1/2) {$\frac{6}{8}$};
    \node at (7*12/8,1/2) {$\frac{7}{8}$};
    \node at (12,1/2) {1};
    \end{tikzpicture}
\end{center}
\label{Di set example m=3}
\end{figure}
    
    We construct the $C_i$ in a very similar vein. 
    We consider the collection $\{ C_1,...,C_{m+1} \}$ and define $C^{\textrm{odd}}$ and $C^{\textrm{even}}$ analogously to $D^{\textrm{odd}}$ and $D^{\textrm{even}}$ above. For each subset in $C^{\textrm{even}}$, we assign a distinct (not assigned to any of the other subsets) interval of size $2^{-m}$ from $\mathcal{D}_m$ and then add this interval to all the $C_i$ in this subset. However, now there are $2^m-1$ non-empty subsets of even size, so there will be an interval of size $2^{-m}$ which will not be in any of the $C_i$. Thus, it follows that:
     \[\mu \left( \bigcup_{i=1}^{m+1} C_i \right) =1-\frac{1}{2^m}.\]

\begin{figure}[h!]
    \caption{An example of the sets $C_i$ for $m=3$.}
    \label{leb construction Ci}
\begin{center}
    \begin{tikzpicture}
    \draw[very thick] (0,5) -- (12,5);
    \draw (0,4) -- (12/8*4,4);
    \draw (0,3) --(12/8*2,3);
    \draw (4*12/8,3) --(6*12/8,3);
    \draw (0,2) -- (12/8,2);
    \draw (2*12/8,2) -- (3*12/8,2);
    \draw (4*12/8,2) -- (5*12/8,2);
    \draw (6*12/8,2) -- (7*12/8,2);
    \draw (0,1) -- (12/8,1);
    \draw (3*12/8,1) -- (4*12/8,1);
    \draw (5*12/8,1) -- (6*12/8,1);
    \draw (6*12/8,1) -- (7*12/8,1);
    \node at (-1/2,5) {[0,1]};
    \node at (-1/2,4) {$C_1$};
    \node at (-1/2,3) {$C_2$};
    \node at (-1/2,2) {$C_3$};
    \node at (-1/2,1) {$C_4$};
    \node at (0,1/2) {0};
    \node at (12/8,1/2) {$\frac{1}{8}$};
    \node at (24/8,1/2) {$\frac{2}{8}$};
    \node at (3*12/8,1/2) {$\frac{3}{8}$};
    \node at (4*12/8,1/2) {$\frac{4}{8}$};
    \node at (5*12/8,1/2) {$\frac{5}{8}$};
    \node at (6*12/8,1/2) {$\frac{6}{8}$};
    \node at (7*12/8,1/2) {$\frac{7}{8}$};
    \node at (12,1/2) {1};
    \end{tikzpicture}
\end{center}
\end{figure}

 Note that conditions $(iii)$ and $(iv)$ from the lemma statement are satisfied by the construction of the sets $C_i$ and $D_i$. Also note that condition $(i)$ follows from condition $(ii)$. Thus, it remains to show that condition $(ii)$ is satisfied. By the symmetry of the construction of the sets $C_i$ and $D_i$, it suffices to show that 
 \[\mu(D_1 \cap ... \cap  D_l) = \mu(C_1 \cap ... \cap  C_l) \qquad \text{for } 1 \leq l \leq m.\]

For a subset $X$ of $[0,1]$, we write $X^c$ to denote its complement, i.e. $X^c=[0,1]\setminus X$. For convenience, it will sometimes be helpful for us to write $X^1$ in place of $X$. We can write $D_1 \cap D_2 \cap \dots \cap D_l$ as the following disjoint union:
 \begin{align*}
     D_1 \cap D_2 \cap \dots \cap D_l &= \bigsqcup_{\substack{\sigma_i \in \{1,c\} \\ l+1 \leq i \leq m+1}} D_1 \cap D_2 \cap ...\cap D_l \cap D_{l+1}^{\sigma_{l+1}} \cap ... \cap D_{m+1}^{\sigma_{m+1}}. 
 \end{align*}
It follows that 
    \begin{align}\label{D decompose}
        \mu(D_1 \cap...\cap D_l)=\sum_{\substack{\sigma_i \in \{1,c\} \\ l+1 \leq i \leq m+1}} \mu\left(D_1\cap...\cap D_l \cap D_{l+1}^{\sigma_{l+1}} \cap ...\cap D_{m+1}^{\sigma_{m+1}} \right).
    \end{align}
We see from the construction of the $D_i$ that for $1 \leq l \leq m$ we have
\begin{equation}\label{D measure}
\mu\left(D_1\cap...\cap D_l \cap D_{l+1}^{\sigma_{l+1}} \cap ... \cap D_{m+1}^{\sigma_{m+1}}\right)=
    \begin{cases}
        \frac{1}{2^m}  & \text{if } ~ l+ \#\{i:\sigma_i=1\} ~\text{is odd,}\\
        0 & \text{else. }
    \end{cases}
\end{equation}
This is because in order for the left-hand side of (\ref{D measure}) to be non-zero we need the number of non-complement sets being included in the intersection to be odd and precisely one interval of length $2^{-m}$ lies in this intersection from the construction.
In a similar vein we can write:
\begin{align*}
     C_1 \cap C_2 \cap \dots \cap C_l &= \bigsqcup_{\substack{\sigma_i \in \{1,c\} \\ l+1 \leq i \leq m+1}} C_1 \cap C_2 \cap ...\cap C_l \cap C_{l+1}^{\sigma_{l+1}} \cap ... \cap C_{m+1}^{\sigma_{m+1}}. 
 \end{align*} 
 Thus,
\begin{align}\label{C decompose}
    \mu\left(C_1 \cap...\cap C_l\right)=\sum_{\substack{\sigma_i \in \{1,c\} \\ l+1 \leq i \leq m+1}} \mu\left(C_1\cap...\cap C_l \cap C_{l+1}^{\sigma_{l+1}} \cap ...\cap C_{m+1}^{\sigma_{m+1}}\right).
\end{align}
We see from the construction of the $C_i$ that for $1 \leq l \leq m$ we have
\begin{equation}\label{C measure}
\mu\left(C_1\cap...\cap C_l \cap C_{l+1}^{\sigma_{l+1}} \cap...\cap C_{m+1}^{\sigma_{m+1}} \right)=
    \begin{cases}
        \frac{1}{2^m}  & \text{if } ~ l+ \#\{i:\sigma_i=1\} ~\text{is even,}\\
        0 & \text{else. }
    \end{cases}
\end{equation}
By similar reasoning to that of $D_i$, in order for the left-hand side of (\ref{C measure}) to be non-zero we need the number of non-complement sets being included in the intersection to be even and precisely one interval of length $2^{-m}$ lies in this intersection from the construction.

 For the set $\{ \sigma_{l+1},...,\sigma_{m+1} \}$, the number of subsets of odd size is equal to the number of subsets of even size, and hence by combining the results $(\ref{D measure})$ and $(\ref{C measure})$ with the decompositions $(\ref{C decompose})$ and $(\ref{D decompose})$, $\mu(C_1\cap...\cap C_l)=\mu(D_1\cap...\cap D_l)$ for $1 \leq l \leq m$ as required.
\end{proof}

We are now ready to complete the proof of Theorem 1.1.

\begin{proof}[Proof of Theorem \ref{the equality theorem}]
We first note that it suffices to prove Theorem \ref{the equality theorem} for the case $c=1$. To see this, suppose that $(A_i)$ and $(B_i)$ are such that the conditions of the theorem are satisfied with $c=1$. For general $c \in [0,1]$ and $X \subset \mathbb{R}$, we define $cX= \{cx: x \in X\}$ to denote the rescaling of $X$ by a factor of $c$. It follows then the sets $(cA_i)$ and $(cB_i)$ will satisfy the conditions of the theorem. Conditions $(i)$, $(ii)$ and $(iv)$ are unchanged by the shrinking and the measure of $\limsup(cB_i)$ is $c$, as desired. \label{c argument}

%and then remark at the end how we can adapt this for any $c \in [0,1]$. We can easily adapt these sets described above to give $\mu(\limsup B_i)=c$ for general $c \in [0,1]$. We take both sequences of $(A_i)_{i \in \mathbb{N}}$ and $(B_i)_{i \in \mathbb{N}}$ defined earlier and multiply each point in each set by $c$, where $c \in [0,1]$, the desired value of $\mu(\limsup B_i)$. We note that the intersection measure equalities in Theorem 
%\ref{the equality theorem} are still satisfied.

    We construct the sequences of sets $(A_i)_{i \in \mathbb{N}}$ and $(B_i)_{i \in \mathbb{N}}$ in blocks. The first block we define will be of size $m+1$, the second block will be of size $(m+1)^2$, the third block will be of size $(m+1)^3$, and so on. The $(m+1)^k$ sets in the $k$th block are constructed from the dyadic intervals of length $2^{-mk}$.

    We proceed by constructing the first $m+1$ terms of $(A_i)_{i \in \mathbb{N}}$ and $(B_i)_{i \in \mathbb{N}}$ from the $C_i$ and~$D_i$ constructed in Lemma \ref{prelim equality lemma}. For $1\leq i \leq m+1$ we set $A_i:=C_i$ and $B_i:=D_i$. Recall from the construction in Lemma \ref{prelim equality lemma} that the $C_i$ and $D_i$ are formed from dyadic intervals of length $2^{-m}$.
    
    To proceed with constructing the next block of sets, we define auxiliary sets $E^2_j$ and $F^2_j$ for $1 \leq j \leq m+1$. We divide $[0,1]$ into dyadic intervals of length $2^{-m}$. For $1 \leq j \leq m+1$ we construct $E^2_{j}$  by shrinking $A_j$ by a factor of $2^{m}$ so that it lies in $[0,1/2^{m}]$ and then place a copy of this in each of the $2^{m}$ dyadic intervals.
    We perform the same procedure to construct $F^2_j$ by shrinking $B_j$ by a factor of $2^{m}$ and then placing a copy of this in each of the $2^{m}$ dyadic intervals. 

    For $1 \leq i \leq m+1$ and $1 \leq j \leq m+1$ we set:
    \begin{align*}
        A_{(m+1)+(m+1)(i-1)+j}:= E^2_j \cap A_i ~,\\
        B_{(m+1)+(m+1)(i-1)+j}:= F^2_j \cap B_i ~.
    \end{align*}
    We further note that this means that the $A_i$ and $B_i$ in the second block are constructed from dyadic intervals of length $2^{-2m}$.

    We continue the process iteratively. For $k \geq 1$, we construct $E^{k+1}_j$ for $1 \leq j \leq m+1$ as follows. We divide $[0,1]$ into dyadic intervals of length $2^{-km}$. Then for $1\leq j \leq m+1$ shrink $A_j$ by a factor of $2^{km}$ and put copies of this in each of the dyadic intervals of length $2^{-km}$. This formed set will be $E^{k+1}_j$. The same procedure is used to construct $F_{j}^{k+1}$, just with $A_j$ replaced with $B_j$. For $1 \leq i \leq (m+1)^k$ and $1 \leq j \leq m+1$ define
    \begin{align}
    \begin{split}\label{preimage of AB construction}
        A_{\sum_{r=1}^{k}(m+1)^r +(m+1)(i-1)+j} :=E_j^{k+1} \cap A_{\sum_{r=1}^{k-1}(m+1)^r+i} ~,\\
        B_{\sum_{r=1}^{k}(m+1)^r +(m+1)(i-1)+j} :=F_j^{k+1} \cap B_{\sum_{r=1}^{k-1}(m+1)^r+i} ~.
        \end{split}
    \end{align}

Once more, we see that the sets in the $(k+1)$th block are built from dyadic intervals of length $2^{-(k+1)m}$. The subscript indices in (\ref{preimage of AB construction}) are rather cumbersome, so it will be convenient to write them more succinctly. To this end, we will define the function $f: \mathbb{N} \rightarrow \mathbb{N}$ that maps the indices of $A$ and $B$ on the left-hand side of \eqref{preimage of AB construction}, respectively, to the indices of $A$ and $B$ on the right-hand side of \eqref{preimage of AB construction}. In essence, to form the $(k+1)$th block of the sequence of $A_i$ we intersect the $E_j^{k+1}$ with the sets in the $k$th block. The function $f$ takes indices in the $(k+1)$th block and maps them to the set in block $k$ that is being intersected with $E_j^k$ to form it. We let $g: \mathbb{N} \rightarrow \mathbb{N}$ map the indices of $A$ and $B$ from the left-hand side of \eqref{preimage of AB construction}, respectively, to the indices of $E^{k+1}$ and $F^{k+1}$ on the right-hand side of \eqref{preimage of AB construction}. Thus, we have
\begin{align}
    \begin{split}\label{function index}
        A_{i} :=E_{g(i)}^{k+1} \cap A_{f(i)} ,\\
        B_{i} :=F_{g(i)}^{k+1} \cap B_{f(i)}. 
        \end{split}
    \end{align}

We give explicit examples of the first 15 terms of the sequences $(A_i)_{i \in \mathbb{N}}$ and $(B_i)_{i \in \mathbb{N}}$ for $m=2$ in Figures \ref{leb construction equal Ai} and \ref{leb construction equal Bi} of the Appendix to aid readers understanding.

The construction ensures that there is a form of independence between the $E_{j}^{k+1}$ and any sets from the first $k$ blocks. The sets in the first $k$ blocks of the sequences $(A_i)_{i \in \mathbb{N}}$ and $(B_i)_{i \in \mathbb{N}}$ are made up of dyadic intervals of length at least $2^{-km}$. Meanwhile, the $E_j^{k+1}$ are just composed of a single pattern in $[0,2^{-km}]$ repeated in each dyadic interval of this length. Due to this finer structure, we observe that for $1 \leq j \leq m+1$ and for any $s\in \mathbb{N}$ and $1 \leq i_1,...,i_s \leq \sum_{r=1}^{k}(m+1)^r$ we have
\begin{align}
\begin{split} \label{EA independence}
    \mu \left( E_j^{k+1} \cap A_{i_1} \cap...\cap A_{i_s}  \right) &= \mu \left(E_j^{k+1} \right) \mu(A_{i_1} \cap...\cap A_{i_s}) \text{ and } \\ 
    \mu \left( F_j^{k+1} \cap B_{i_1} \cap ... \cap B_{i_s}   \right) &= \mu \left(F_j^{k+1} \right) \mu( B_{i_1} \cap ... \cap B_{i_s}) .
    \end{split}
\end{align}

    We now show that the desired measure theoretic properties of Theorem $\ref{the equality theorem}$ are satisfied. We split this into two parts, first showing that properties $(i)$ and $(ii)$ involving the measures of sets and intersections hold. We then show that properties $(iii)$ and $(iv)$ relating to the measure of limsup sets also hold.
    
    We note that $(ii)$ implies $(i)$ and so it suffices to prove the former, i.e. for any sequence of natural numbers $i_1,i_2,...,i_l$ with $1 \leq l \leq m$, we have 
     \[\mu(A_{i_1} \cap ... \cap A_{i_l})  = \mu(B_{i_1} \cap ... \cap B_{i_l}).\]
    We prove $(ii)$ using an induction procedure on the number of blocks and show that the measure relations hold for sets in these blocks. Specifically, we induct on $v \in \mathbb{N}$ in the following statement:
    
    For $1 \leq l \leq m$ and indices $ 1 \leq i_1<...<i_l \leq \sum_{r=1}^v (m+1)^r$ we have
\begin{align}\label{induction AB statement}
         \mu(A_{i_1} \cap ... \cap A_{i_l}) & = \mu(B_{i_1} \cap ... \cap B_{i_l}). 
\end{align}
We begin with the basis step $v=1$. The $A_i$ and $B_i$ in the first block of size $m+1$ are the same as the $C_i$ and $D_i$ of Lemma \ref{prelim equality lemma}. Indeed, the relationship of the measures of the sets and their intersections in Lemma $\ref{prelim equality lemma}$ ensure that (\ref{induction AB statement}) is satisfied for $v=1$.

    Continuing with the induction, assume that \eqref{induction AB statement} is true for $v=k$. We will show that~\eqref{induction AB statement} is then true for $v=k+1$. Suppose that we are given $1\leq i_1<...<i_l\leq \sum_{r=1}^{k+1}(m+1)^r$ with $i_t,...,i_l$ in block $k+1$ for some $1 \leq t \leq l$. Indeed, if none of the $i_j$ are in the $(k+1)$th block, then we have already proved the statement by the induction hypothesis. 
    It follows from \eqref{preimage of AB construction} and \eqref{EA independence} that
    \begin{align*}
        \mu(A_{i_1} \cap ... \cap A_{i_l})&=\mu(A_{i_1} \cap ... \cap A_{i_{t}} \cap ... \cap A_{i_l}) \\
        &=\mu\left(A_{i_1} \cap ... \cap {A}_{f(i_t)}\cap {E}_{g(i_t)}^{k+1} \cap ... \cap {A}_{f(i_l)}\cap {E}_{g(i_l)}^{k+1} \right) \\
        &=\mu\left(A_{i_1} \cap ... \cap {A}_{f(i_t)} \cap ... \cap {A}_{f(i_l)}\right) \mu\left({E}_{g(i_t)}^{k+1} \cap ... \cap {E}_{g(i_r)}^{k+1} \right).
    \end{align*}
    %Where from (\ref{preimage of AB construction}) we have each set in block $k+1$, can be written as a set in block $k$ intersected with another set from $E^{k+1}$ and use the functions $f$ and $g$ to write it in this form. By reasoning from (\ref{EA independence}) we have that the intersection of sets in the first $k$ blocks will be independent of intersections of $E^{k+1}$ sets and thus we have factorising on the third line. 
    Similarly,
    \begin{align*}
        \mu(B_{i_1} \cap ... \cap B_{i_l})=\mu\left(B_{i_1} \cap ... \cap {B}_{f(i_t)} \cap ... \cap {B}_{f(i_l)}\right) \mu\left({F}_{g(i_t)}^{k+1} \cap ... \cap {F}_{g(i_l)}^{k+1} \right).
    \end{align*}
    From the inductive hypothesis, we have 
    \[\mu\left(B_{i_1} \cap ... \cap {B}_{f(i_t)} \cap ... \cap {B}_{f(i_l)}\right)= \mu\left(A_{i_1} \cap ... \cap {A}_{f(i_t)} \cap ... \cap {A}_{f(i_l)} \right),\] as all the sets being intersected are in the first $k$ blocks. Furthermore, Lemma \ref{prelim equality lemma} ensures that we have $\mu\left({E}_{g(i_t)}^{k+1} \cap ... \cap {E}_{g(i_l)}^{k+1} \right)=\mu\left({F}_{g(i_t)}^{k+1} \cap ... \cap {F}_{g(i_l)}^{k+1} \right)$ since these sets are generated from the shrunken and repeated $C_i$ and $D_i$.
    Hence, \eqref{induction AB statement} is true for $v=k+1$. It follows by induction that \eqref{induction AB statement} holds for all $v \in \mathbb{N}$. This concludes the proof that properties $(i)$ and $(ii)$ of the statement of the theorem hold.

We now show that the limsup sets have the desired measures, as per properties $(iii)$ and~$(iv)$.
%We note here that the fact that we define $A_i$ in the $k+1$th block to be the $E^{k+1}$ sets intersected with $A_i$ from block $k$ forces later blocks to be nested within the previous blocks. 
By construction, the union of the $A_i$'s in each block is nested in the union of the $A_i$'s in the previous block. More precisely, for every $k \geq 1$, we have 
     \begin{align*}
        \bigcup_{l=\sum_{r=1}^{k} (m+1)^{r}+1}^{\sum_{r=1}^{k+1} (m+1)^{r}} A_l &= \bigcup_{i=1}^{(m+1)^{k}} \bigcup_{j=1}^{m+1} E_j^{k+1} \cap A_{\sum_{r=1}^{k-1} (m+1)^r+i}  \\
        &\subseteq \bigcup_{i=1}^{(m+1)^{k}} A_{\sum_{r=1}^{k-1}(m+1)^r+i} ~.
    \end{align*}
    So we have the following chain of inclusions:
    \begin{align} \label{nesting}
        \bigcup_{l=\sum_{r=1}^{k} (m+1)^{r}+1}^{\sum_{r=1}^{k+1} (m+1)^{r}} A_l \subseteq \bigcup_{l=\sum_{r=1}^{k-1} (m+1)^{r}+1}^{\sum_{r=1}^{k}(m+1)^{r}} A_l \subseteq  ... \subseteq \bigcup_{l=1}^{m+1} A_l ~.%\\
        %\bigcup_{l=(m+1)^{k+1}+1}^{(m+1)^{k+2}} B_l \subset \bigcup_{l=(m+1)^{k}+1}^{(m+1)^{k+1}} B_l \subset  ... \subset \bigcup_{l=1}^{m+1} B_l
    \end{align}

Furthermore, noting that $A_1,...,A_{m+1}$ are equal to the $C_1,...C_{m+1}$ from Lemma \ref{prelim equality lemma}, we have that
\begin{align*}
    \mu\left( \bigcup_{i=1}^{m+1} A_i \right)=\mu\left( \bigcup_{i=1}^{m+1} C_i \right)=1-\frac{1}{2^m}.
\end{align*}

 We also recall that the $E^k_j$ sets are constructed from shrunk versions of the $C_i$. This ensures that 
 the union of $E^k_j$ over $1 \leq j\leq m+1$ covers $1-\frac{1}{2^m}$ of each dyadic interval of width $2^{-(k-1)m}$. This ensures, for each $k \in \mathbb{N}$,
\begin{align} \label{measure of E}
\mu\left(\bigcup_{j=1}^{m+1}{E_j^k}\right) &= 1-\frac{1}{2^m}.    
\end{align}
It follows from \eqref{measure of E} combined with \eqref{EA independence} that
 \begin{align*}
    \mu \left(  \bigcup_{l=\sum_{r=1}^{k} (m+1)^{r}+1}^{\sum_{r=1}^{k+1} (m+1)^{r}} A_l \right) &= \mu \left( \bigcup_{j=1}^m  E_j^{k+1} \cap  \bigcup_{i=1}^{(m+1)^{k}} A_{\sum_{r=1}^{k-1} (m+1)^r+i}  \right) \\
    &= \mu \left( \bigcup_{j=1}^m  E_j^{k+1} \right) \left(  \bigcup_{i=1}^{(m+1)^{k}} A_{\sum_{r=1}^{k-1} (m+1)^r+i}  \right) \\
    &=\left( 1-\frac{1}{2^m} \right) \left( \bigcup_{l=\sum_{r=1}^{k-1} (m+1)^{r}+1}^{\sum_{r=1}^{k}(m+1)^{r}} A_l  \right) = ...=\left( 1-\frac{1}{2^m} \right)^{k+1}.
    \end{align*}
It then follows from the fact that the blocks of $A_i$'s are nested (see \eqref{nesting}) that property $(iii)$ is true that \[\mu(\limsup_{i \rightarrow \infty}A_i)=0.\]

Next recall that for $1 \leq i \leq m+1$ we have that $B_i$ is equal to $D_i$ from Lemma \ref{prelim equality lemma}. Moreover, recall that 
\[ \bigcup_{i=1}^{m}{D_i} = [0,1].\]
We can inductively show that the union of the $B_i$'s in the $k$-th block also fully covers the unit interval. To see this, recall that for each $k \geq 2$ the $F_j^k$ are a collection of shrunk down and repeated $D_j$. Consequently,
\[\bigcup_{j=1}^{m+1}{F_j^k} = [0,1].\]
It follows that for all $k \geq 2$ we have 
    \begin{align*}
        \bigcup_{l=\sum_{r=1}^{k} (m+1)^{r}+1}^{\sum_{r=1}^{k+1} (m+1)^{r}} B_l &= \bigcup_{i=1}^{(m+1)^{k}} \bigcup_{j=1}^{m+1} F_j^{k+1} \cap B_{\sum_{r=1}^{k-1} (m+1)^r+i}  \\
        &=\bigcup_{i=1}^{(m+1)^{k}} B_{\sum_{r=1}^{k-1}(m+1)^r+i} \\
        &= \bigcup_{l=\sum_{r=1}^{k-1}(m+1)^r +1}^{\sum_{r=1}^{k}(m+1)^r} B_l ~.
    \end{align*}
Therefore, since $\bigcup_{i=1}^{m+1}{B_i} = [0,1]$, it follows by induction that
\[\bigcup_{l=\sum_{r=1}^{k-1} (m+1)^{r}+1}^{\sum_{r=1}^{k} (m+1)^{r}} B_l = [0,1]\]
for all $k \geq 2$. 

For a natural number $h \geq 2$, write
\[\mathcal{B}_h = \bigcup_{l=\sum_{r=1}^{h-1} (m+1)^{r}+1}^{\sum_{r=1}^{h} (m+1)^{r}} B_l.\]
It follows that $\mu{(\mathcal{B}_h)}=1$ for all $h \in \mathbb{N}$ and hence $\mu(\limsup_{h \to \infty}{\mathcal{B}_h}) = 1$. Since 
\[\limsup_{h \to \infty}{\mathcal{B}_h}  =\limsup_{i \to \infty} B_i\] we are done showing property $(iv)$. This completes the proof for $c=1$. 
\end{proof}

\section{Proof of Theorem \ref{the big theorem}} \label{the k intersection case}

    We begin the section by defining sets $(G^{p,q}_n)_{n \in \mathbb{N}}$ and $(H^{p,q}_n)_{n \in \mathbb{N}}$ where $p,q \in \mathbb{N}$ with $p \leq q$ which will have very convenient measure properties. These will be used to construct the sequences $(A_i)_{i \in \mathbb{N}}$ and $(B_i)_{i \in \mathbb{N}}$ of Theorem \ref{the big theorem}. The remainder of the chapter is spent checking that the sequences $(A_i)_{i \in \mathbb{N}}$ and $(B_i)_{i \in \mathbb{N}}$ satisfy properties $(i)-(iv)$ of the theorem.

\subsection{Preliminary Lemmas}

\begin{lemma}\label{construction of Gpq sets}
    Let $\mu$ denote the Lebesgue measure on $[0,1]$. For $p,q \in \mathbb{N}$ with $p \leq q$, there exist sequences $(G^{p,q}_i)_{i \in \mathbb{N}}$ and $(H^{p,q}_i)_{i \in \mathbb{N}}$ of Borel subsets of $[0,1]$ with the following properties:
\begin{enumerate}[label=(\roman*)]
    \item $G_i^{p,q} \subseteq H_i^{p,q}$ for all $i \in \mathbb{N}$,
    \item $ \mu\left(G_i^{p,q}\right) = \frac{p}{qi} $ for all $i \in \mathbb{N}$, 
    \item$ \mu(H_i^{p,q}) = \frac{1}{i} $ for all $i \in \mathbb{N}$, 
    \item  $ H_{i}^{p,q}\subseteq H_{i+1}^{p,q}$ for all $i \in \mathbb{N}$, 
     \item For any finite sequence of natural numbers $1<i_1<i_2<...<i_n$, we have 
     \[\mu\left(G_{i_1}^{p,q} \cap G_{i_2}^{p,q} \cap ... \cap G_{i_n}^{p,q}\right) = \frac{p^n}{q^n i_n},\]
     \item $\mu\left(\limsup_{i \rightarrow \infty}(H^{p,q}_i)\right)=\mu\left(\limsup_{i \rightarrow \infty}(G^{p,q}_i)\right)=0$.
    \end{enumerate}

\end{lemma}

\begin{proof}
We iteratively construct both sequences simultaneously in such a way that properties $(i)-(iv)$ will be immediate from the construction. We begin by setting $H^{p,q}_1=[0,1]$ and $G^{p,q}_1=\left[0,\frac{p}{q}\right]$. 

To construct the set $H_2^{p,q}$, we take $H^{p,q}_1=[0,1]$ and place $q$ disjoint intervals of measure $\frac{1}{2q}$ within it so that the total measure is $\frac{1}{2}$. Explicitly, we can do this as follows:
\[H^{p,q}_2=\left[0,\frac{1}{2q}\right] \cup \left[\frac{1}{q}, \frac{3}{2q}\right] \cup\left[\frac{2}{q}, \frac{5}{2q}\right] \cup ... \cup \left[\frac{q-1}{q},\frac{2q-1}{2q}\right].\]
We then construct $G^{p,q}_2$ from $H^{p,q}_2$ simply by selecting the first $p$ intervals and discarding the rest. This ensures that we have $\mu(G_2^{p,q})=\frac{p}{q}\mu(H_2^{p,q})=\frac{p}{2q}$.

To construct the set $H^3_{p,q}$, we use $q^2$ intervals of length $\frac{1}{3q^2}$. More precisely, we place $q$ disjoint intervals of length $\frac{1}{3q^2}$ within each of the disjoint intervals of $H_2^{p,q}$ in a similar fashion to the construction of $H^2_{p,q}$ above. This ensures that the measure of $H^{p,q}_3$ is $\frac{1}{3}$.

To construct $G^{p,q}_3$ from $H_3^{p,q}$, we simply go along and include the first $p$ intervals, then skip the next $q-p$, then include the next $p$, then skip the next $q-p$ and so on. This ensures that there are exactly $p$ intervals of $G^{p,q}_3$ contained within each interval of $H_2^{p,q}$ and, furthermore, $\mu(G^{p,q}_3)=\frac{p}{q} \mu(H^{p,q}_3)=\frac{p}{3q}$.

We construct subsequent sets inductively.
Having constructed $H_{n}^{p,q}$ from $q^{n-1}$ intervals of length $\frac{1}{nq^{n-1}}$ we then construct $H_{n+1}^{p,q}$ by placing $q$ disjoint intervals of length $\frac{1}{(n+1)q^n}$ within each of the intervals of $H_n^{p,q}$. We deduce that the measure of $H_{n+1}^{p,q}$ is $q^{n-1} \times q \times \frac{1}{(n+1)q^n}= \frac{1}{n+1}$. 

Next, for the construction of $G_{n+1}^{p,q}$ we go along the intervals of $H_{n+1}^{p,q}$ and select the first $p$ to be included in $G_{n+1}^{p,q}$, then reject the next $q-p$, then include the following $p$, then reject the next $q-p$ and so on. This implies $\mu(G_{n+1}^{p,q})=\frac{p}{q(n+1)}$. Furthermore, the construction ensures that any interval of $H_n^{p,q}$ will contain exactly $p$ intervals of $G_{n+1}^{p,q}$. This completes the construction.

\begin{figure}[h]
    \caption{An example of the $G_n^{3,5}$ construction }
    \label{Gpq 3 5}
\begin{center}
    \begin{tikzpicture}
    \draw[very thick] (0,5) -- (12,5);
    \draw (0,4) -- (12,4);
    \draw (0,3) -- (12*3/5,3);
    \draw (0,2) -- (12/10,2) (0+12/5,2) -- (12/10+12/5,2) (0+2*12/5,2) -- (12/10+2*12/5,2) (0+3*12/5,2) -- (12/10+3*12/5,2) (0+4*12/5,2) -- (12/10+4*12/5,2);
    \draw (0,1) -- (12/10,1) (0+12/5,1) -- (12/10+12/5,1) (0+2*12/5,1) -- (12/10+2*12/5,1);

    \draw (0,0) -- (12/75,0) (12/50,0) -- (12/50+12/75,0) (2*12/50,0) -- (2*12/50+12/75,0) (3*12/50,0) -- (3*12/50+12/75,0) (4*12/50,0) -- (4*12/50+12/75,0)
    (12/5,0) -- (12/5+12/75,0) (12/5+12/50,0) -- (12/5+12/50+12/75,0) (2*12/50+12/5,0) -- (2*12/50+12/5+12/75,0) (3*12/50+12/5,0) -- (3*12/50+12/5+12/75,0) (4*12/50+12/5,0) -- (4*12/50+12/5+12/75,0)
    (2*12/5,0) -- (2*12/5+12/75,0) (2*12/5+12/50,0) -- (2*12/5+12/50+12/75,0) (2*12/5+2*12/50,0) -- (2*12/5+2*12/50+12/75,0) (2*12/5+3*12/50,0) -- (2*12/5+3*12/50+12/75,0) (2*12/5+4*12/50,0) -- (2*12/5+4*12/50+12/75,0) (3*12/5,0) -- (3*12/5+12/75,0) (3*12/5+12/50,0) -- (3*12/5+12/50+12/75,0) (3*12/5+2*12/50,0) -- (3*12/5+2*12/50+12/75,0) (3*12/5+3*12/50,0) -- (3*12/5+3*12/50+12/75,0) (3*12/5+4*12/50,0) -- (3*12/5+4*12/50+12/75,0)
    (4*12/5,0) -- (4*12/5+12/75,0) (4*12/5+12/50,0) -- (4*12/5+12/50+12/75,0) (4*12/5+2*12/50,0) -- (4*12/5+2*12/50+12/75,0) (4*12/5+3*12/50,0) -- (4*12/5+3*12/50+12/75,0) (4*12/5+4*12/50,0) -- (4*12/5+4*12/50+12/75,0);

\draw (0,-1) -- (12/75,-1) (12/50,-1) -- (12/50+12/75,-1) (2*12/50,-1) -- (2*12/50+12/75,-1) 
    (12/5,-1) -- (12/5+12/75,-1) (12/5+12/50,-1) -- (12/5+12/50+12/75,-1) (2*12/50+12/5,-1) -- (2*12/50+12/5+12/75,-1) 
    (2*12/5,-1) -- (2*12/5+12/75,-1) (2*12/5+12/50,-1) -- (2*12/5+12/50+12/75,-1) (2*12/5+2*12/50,-1) -- (2*12/5+2*12/50+12/75,-1) 
    (3*12/5,-1) -- (3*12/5+12/75,-1) (3*12/5+12/50,-1) -- (3*12/5+12/50+12/75,-1) (3*12/5+2*12/50,-1) -- (3*12/5+2*12/50+12/75,-1) 
    (4*12/5,-1) -- (4*12/5+12/75,-1) (4*12/5+12/50,-1) -- (4*12/5+12/50+12/75,-1) (4*12/5+2*12/50,-1) -- (4*12/5+2*12/50+12/75,-1);

    \node at (-1/2,5) {[0,1]};
    \node at (-1/2,4) {$H_1^{3,5}$};
    \node at (-1/2,3) {$G_1^{3,5}$};
    \node at (-1/2,2) {$H_2^{3,5}$};
    \node at (-1/2,1) {$G_2^{3,5}$};
    \node at (-1/2,0) {$H_3^{3,5}$};
    \node at (-1/2,-1) {$G_3^{3,5}$};
    \end{tikzpicture}
\end{center}
\end{figure}

    We now check that the measure theoretic properties $(v)$ and $(vi)$ hold. Regarding $(v)$, this follows from the fact that for $n>1$ each interval of $H_n^{p,q}$ contains equally many ($q$ many to be precise) intervals of $H_{n+1}^{p,q}$. Furthermore, the construction ensures that each interval of $H_n^{p,q}$ (and subsequently $G_n^{p,q}$) contains equally many ($p$ many to be exact) intervals of $G_{n+1}^{p,q}$. Therefore, when we intersect $G_n^{p,q}$ with an earlier set in the sequence $(G_i^{p,q})_{i \in \mathbb{N}}$, we are left with $p/q$ of the intervals of $G_n^{p,q}$. We can inductively see that intersecting $G_n^{p,q}$ with $k < n$ distinct earlier sets in the $(G_i^{p,q})_{i \in \mathbb{N}}$ sequence leaves us with $(p/q)^k$ of the intervals of $G_n^{p,q}$. As the intervals are of equal length, this proves statement $(v)$.

To show $(vi)$, that the measure of the limsup sets is $0$, we first note that it suffices to show the $\mu\left(\limsup_{n \rightarrow \infty} H_{n}^{p,q}\right)=0$. This is because from $(i)$, $G_n^{p,q} \subseteq H_n^{p,q}$ for all~$n \in \mathbb{N}$ and thus $\limsup_{n \rightarrow \infty} G_n^{p,q} \subseteq \limsup_{n \rightarrow \infty} H_n^{p,q}$. That $\mu(\limsup_{n \rightarrow \infty} H_n^{p,q})=0 $ follows immediately from properties $(iii)$ and $(iv)$. \qedhere

% \begin{align*}
% \mu(\limsup_{n \rightarrow \infty} (G_n^{p,q})) &\leq \mu(\limsup_{n \rightarrow \infty} (H_n^{p,q})) \\
% &\leq \mu(H_n^{p,q})= \frac{1}{n} \hspace{20pt} \forall n
% \end{align*}

\end{proof}

\subsection{Proof of Theorem \ref{the big theorem}}

We first note that it suffices to prove Theorem \ref{the big theorem} for the case $c=1$ by the same reasoning as in the proof of Theorem \ref{the equality theorem}, see page \pageref{c argument}.
\subsubsection{Construction of the sequence \texorpdfstring{$(A_i)_{i \in \mathbb{N}}$}{Ai}}

To construct the sequences of sets needed to prove Theorem \ref{the big theorem} we scale down some of the \(G_{n}^{p,q}\) sets from Lemma \ref{construction of Gpq sets}. After shrinking the sets, we arrange multiple shifted versions of them such that all the \(G_{n}^{p,q}\) sets are disjoint. Specifically, to prove Theorem \ref{the big theorem} for $m \in \mathbb{N}$, we select $(A_i)_{i \in \mathbb{N}}$ to be of the form
\[A_i= c_1 G_{i}^{p_1,q_1} \cup ( c_2 G_{i}^{p_2,q_2}+d_2 ) \cup (c_3 G_{i}^{p_2,q_3}+d_3)  \cup ... \cup (c_m G_{i}^{p_m,q_m}+d_m), \]
where the $p_i$ and $q_i$ will be chosen later to ensure the inequalities of Theorem \ref{the big theorem} hold true, $c_i$ are the shrinking factors that will be selected later, and $d_i$ are the shifts chosen to ensure that there is no overlap between the distinct $G_i^{p,q}$. So we have $d_2>c_1$ and for $2 < i \leq m$ we have $d_i \geq d_{i-1} +c_{i-1}$. Furthermore, we require $A_i$ to lie in the unit interval, so $c_m+d_m \leq 1$. 
We now examine the measures of these sets and the measures of their intersections also.

\begin{lemma}\label{Ai lemma}
    The sequence $(A_i)_{i \in \mathbb{N}}$ described above satisfies: 
    \begin{enumerate}[label=(\roman*)]
    \item For all $i \in \mathbb{N}$,
    \[ \mu(A_i) = c_1 \frac{p_1}{i q_1} + c_2 \frac{p_2}{i q_2} + ... + c_m \frac{p_m}{i q_m}.\]
     \item For any sequence of natural numbers $1<i_1<i_2<...<i_r$ we have
     \[\mu(A_{i_1} \cap A_{i_2} \cap ... \cap A_{i_r}) =  c_1 \frac{p_1^r}{i_r q_1^r} + c_2 \frac{p_2^r}{i_r q_2^r} + ... + c_k  \frac{p_m^r}{i_m q_m^r}.\]
     \item $\mu(\limsup_{i \rightarrow \infty} A_i) =0$.
    \end{enumerate}

\end{lemma}

 \begin{proof}
 The results follow immediately from Lemma \ref{construction of Gpq sets} $(ii), (v)$ and $(vi)$ and the separation conditions imposed by the choices of $c_i$ and $d_i$. \qedhere
% We go through each of the rows of the lemma and use Lemma \ref{construction of Gpq sets} to establish the desired equalities.
% \begin{align*}
% \mu(A_i) &= \mu(c_1 G_{i}^{p_1,q_1} \cup ( c_2 G_{i}^{p_2,q_2}+d_2 ) \cup (c_3 G_{i}^{p_2,q_3}+d_3)  \cup ... \cup (c_k G_{i}^{p_k,q_k}+d_k)) \\
% &= \mu(c_1 G_{i}^{p_1,q_1}) + \mu( c_2 G_{i}^{p_2,q_2}) + \mu( c_3 G_{i}^{p_2,q_3} ) + ... + \mu( c_k G_{i}^{p_k,q_k})) \\
% &=c_1\mu(G_{i}^{p_1,q_1}) + c_2 \mu(G_{i}^{p_2,q_2}) + c_3 \mu( G_{i}^{p_2,q_3}) + ... + c_k \mu( G_{i}^{p_k,q_k}) \\
% &=c_1 \frac{p_1}{i q_1} + c_2 \frac{p_2}{i q_2} + c_3 \frac{p_3}{i q_3} + ... + c_k \frac{p_k}{i q_k}
% \end{align*}

% In a similar vein we find that for $i<j$:
% \begin{align*}
%     \mu(A_{i} \cap A_{j}) &= c_1\mu(G_{i}^{p_1,q_1} \cap G_{j}^{p_1,q_1}) + c_2 \mu(G_{i}^{p_2,q_2} \cap G_{j}^{p_2,q_2}) +  ... + c_k \mu( G_{i}^{p_k,q_k} \cap G_{j}^{p_k,q_k}) \\
%     &= c_1 \frac{p_1^2}{j q_1^2} + c_2 \frac{p_2^2}{j q_2^2} + ... + c_k  \frac{p_k^2}{j q_k^2}
% \end{align*}

% Lastly for $i_1<i_2<...<i_r$
% \begin{align*}
%     \mu(A_{i_1} \cap A_{i_2} \cap ... \cap A_{i_r}) =  c_1 \frac{p_1^r}{i_r q_1^r} + c_2 \frac{p_2^r}{i_r q_2^r} + ... + c_k  \frac{p_k^r}{i_r q_k^r}
% \end{align*}
% It is finally worth noting that the limsup $A_n$ has measure 0. As after all, it is the disjoint combination of sequences $G_n^{p,q}$ , so it suffices to look at the limsup measure of each $G^{p,q}_n$. This was shown to be zero in Proposition \ref{construction of Gpq sets}.
\end{proof}

\subsubsection{Construction of the sequence \texorpdfstring{$(B_i)_{i \in \mathbb{N}}$}{Bn}}

We perform a similar operation with the sequence $(B_i)_{i \in \mathbb{N}}$, but now we incorporate a `floating bit', which we will denote as $K_i^\delta$. It will first be convenient to define the intermediary sets $(\Tilde{B})_{i \in \mathbb{N}}$ to be of the form
\[\Tilde{B_i}= \Tilde{c}_1 G_{i}^{p_1,q_1} \cup ( \Tilde{c}_2 G_{i}^{p_2,q_2}+ \Tilde{d}_2 ) \cup (\Tilde{c}_3 G_{i}^{p_3,q_3}+\Tilde{d}_3)  \cup ... \cup (\Tilde{c}_m G_{i}^{p_m,q_m}+\Tilde{d}_m ),\]
and select $(B_i)_{i \in \mathbb{N}}$ to be of the form
\[B_i= \Tilde{B}_i \cup  K_i^{\delta}  .\]
Again, $p_i$ and $q_i$, are as before for $A_i$. We will choose  $\Tilde{c}_i$ later in the proof, and the shifts $\Tilde{d}_i$ are chosen to ensure the same separation criteria as for $c_i$ and $d_i$ in the construction of $A_i$. The `floating bit' $K_i^\delta$ is essentially a `small' (relative to $\Tilde{c}_i G_i^{p_i,q_i}$) set that shifts position with $i$ to ensure that the measure of the limsup of the sequence equals~$1$.

We construct $K_i^{\delta}$ iteratively. First, we define a preliminary sequence of sets $\left(I_i^\delta\right)_{i \in \mathbb{N}}$. By setting $I_1^\delta:=[0,\delta]$ and for $n>1$ setting
\begin{align*}
    I_n^\delta:=\left[\left( \sum_{i=1}^{n-1} \frac{\delta}{i}\right) \mod{1},\left( \sum_{i=1}^{n} \frac{\delta}{i}\right) \mod{1}\right].
\end{align*}
If the left of the interval is greater than the right, say $[x,y]$ for $x>y$ we use the convention that this means $[0,y] \cup [x,1]$. Another way of viewing the construction of the sets is that having defined $I_{n-1}^\delta$, we continue to define $I_n^\delta$ by placing an interval where the previous $I_{n-1}^\delta$ ended (resuming from 0 after reaching 1) in such a way that the total measure of the set $I_{n}^\delta$ is $\frac{\delta}{n}$. We define $K_i^\delta$ by taking $I_i^\delta$ and removing any points that intersect with any $H_i^{p,q}$ corresponding to $G_i^{p,q}$ in $\Tilde{B}_i$. We can write this explicitly by letting 
\begin{align*}
    \Tilde{B}_i^H:= \Tilde{c}_1 H_{i}^{p_1,q_1} \cup ( \Tilde{c}_2 H_{i}^{p_2,q_2}+ \Tilde{d}_2 ) \cup ... \cup (\Tilde{c}_m H_{i}^{p_m,q_m}+\Tilde{d}_m ),
\end{align*}
and setting,
\begin{align*}
    K_i^\delta:= I_i^\delta \backslash \Tilde{B}^H_i
\end{align*}
It follows from the definition of $K_i^\delta$ and Lemma $\ref{construction of Gpq sets}$ $(i)$ and $(iv)$, that when we intersect $B_m$ with a later set in the sequence, say $B_n$ for $n>m$, the $K_m^{\delta}$ component of $B_m$ does not intersect any of the $G_n^{p,q}$ components of $B_n$. Defining $K_i^{\delta}$ as we did ensures $\mu(\limsup_{i \rightarrow\infty}B_i)=1$, as will be proven in the following lemma.

\begin{lemma}\label{Bi lemma}
    The sequence $(B_i)_{i \in \mathbb{N}}$ described above satisfies: 
    \begin{enumerate}[label=(\roman*)]
    \item for all $i \in \mathbb{N}$,
    \[ \mu(B_i) \geq \Tilde{c}_1 \frac{p_1}{i q_1} + \Tilde{c}_2 \frac{p_2}{i q_2} + \Tilde{c}_3 \frac{p_3}{i q_3} + ... + \Tilde{c}_m \frac{p_m}{i q_m}, \text{and} \]
    \[\mu(B_i) \leq \Tilde{c}_1 \frac{p_1}{i q_1} + \Tilde{c}_2 \frac{p_2}{i q_2} + \Tilde{c}_3 \frac{p_3}{i q_3} + ... + \Tilde{c}_m \frac{p_m}{i q_m} + \frac{\delta}{i}.\]
     \item For any sequence of natural numbers $1<i_1<i_2<...<i_r$ we have
     \[\mu(B_{i_1} \cap B_{i_2} \cap ... \cap B_{i_r}) \geq \Tilde{c}_1 \frac{p_1^r}{i_r q_1^r} + \Tilde{c}_2 \frac{p_2^r}{i_r q_2^r} + ... + \Tilde{c}_m \frac{p_m^r}{i_r q_m^r}, \text{and}\]
     \[\mu(B_{i_1} \cap B_{i_2} \cap ... \cap B_{i_r}) \leq \Tilde{c}_1 \frac{p_1^r}{i_r q_1^r} + \Tilde{c}_2 \frac{p_2^r}{i_r q_2^r} + ... + \Tilde{c}_m \frac{p_m^r}{i_r q_m^r} + \frac{\delta}{i_r}.\]
     \item $\mu(\limsup_{i\rightarrow \infty}B_i) =1$.
    \end{enumerate}

\end{lemma}

\begin{proof}
The proof is very similar to Lemma $\ref{Ai lemma}$ for the sequence $(A_i)_{i \in \mathbb{N}}$. However, the $K_i^{\delta}$ term poses an additional challenge when trying to compute measures of $(B_i)_{i \in \mathbb{N}}$ and their intersections. Consequently, in this case we are able to obtain bounds on the measures rather than precise formulae (as in Lemma \ref{Ai lemma}). The lower bounds in $(i)$ and $(ii)$ are trivial as we just neglect the $K_i^{\delta}$ terms and use Lemma \ref{construction of Gpq sets} $(ii)$ and $(v)$. The upper bound in $(i)$ is also trivial from subadditivity and the fact that the measure of $K_i^\delta$ is at most
$\frac{\delta}{i}$. From the construction of $K_i^\delta$, it does not intersect any of the corresponding $H_i^{p,q}$ sets and therefore, from Lemma \ref{construction of Gpq sets} $(i)$ and $(iv)$ it will not intersect the corresponding $G_m^{p,q}$-sets for $m \geq i$ and so the only possible contribution of the $K^\delta$ sets to the intersection is from $K_{i_r}^\delta$. This, combined with Lemma \ref{construction of Gpq sets} gives us the upper bound in $(ii)$.

% \begin{align*}
%     \mu(B_i) &\geq \Tilde{c}_1 \frac{p_1}{i q_1} + \Tilde{c}_2 \frac{p_2}{i q_2} + \Tilde{c}_3 \frac{p_3}{i q_3} + ... + \Tilde{c}_k \frac{p_k}{i q_k} \\
% \mu(B_i) &\leq \Tilde{c}_1 \frac{p_1}{i q_1} + \Tilde{c}_2 \frac{p_2}{i q_2} + \Tilde{c}_3 \frac{p_3}{i q_3} + ... + \Tilde{c}_k \frac{p_k}{i q_k} + \frac{\delta}{i} \\
% \mu(B_{i_1} \cap B_{i_2} \cap ... \cap B_{i_r}) &\geq \Tilde{c}_1 \frac{p_1^r}{i_r q_1^r} + \Tilde{c}_2 \frac{p_2^r}{i_r q_2^r} + ... + \Tilde{c}_k \frac{p_k^r}{i_r q_k^r} \\
% \mu(B_{i_1} \cap B_{i_2} \cap ... \cap B_{i_r}) &\leq \Tilde{c}_1 \frac{p_1^r}{i_r q_1^r} + \Tilde{c}_2 \frac{p_2^r}{i_r q_2^r} + ... + \Tilde{c}_k \frac{p_k^r}{i_r q_k^r} + \frac{\delta}{i_r}
% \end{align*}

To show $(iii)$ we note $I_i^{\delta}$ cycles through the unit interval infinitely often since the sum $\sum_{n=1}^\infty \frac{\delta}{n}$ is infinite. Hence $\mu(\limsup_{i \rightarrow \infty}I_i^\delta) =1$. The stipulation that the $K_i^\delta$ sets avoid $H_i^{p,q}$ still ensures $\mu(\limsup_{i\rightarrow \infty} K_i^\delta)=1$ too because the $H_i^{p,q}$ are nested and shrink to a set of measure 0 and so $K_i^\delta$ will still cycle through almost every point. This can be explicitly seen as follows:
\begin{align*}
    \mu\left( \limsup_{i \rightarrow \infty} K_i^\delta \right)&=\mu\left( \bigcap_{n \geq 1}\bigcup_{i \geq n} K_i^\delta\right) \\ &=     \mu\left( \bigcap_{n \geq 1}\bigcup_{i \geq n} I_i^\delta \backslash \Tilde{B}_i^H \right)\\ &=  \mu\left( \bigcap_{n \geq k}\bigcup_{i \geq n} I_i^\delta \backslash \Tilde{B}_i^H \right)\\ &\geq
    \mu\left( \left(\bigcap_{n \geq k}\bigcup_{i \geq n} I_i^\delta\right) \backslash \Tilde{B}_k^H\right) \\
    &=\mu \left([0,1]\backslash(\Tilde{c}_1 H_{k}^{p_1,q_1} \cup ( \Tilde{c}_2 H_{k}^{p_2,q_2}+ \Tilde{d}_2 ) \cup ... \cup (\Tilde{c}_m H_{k}^{p_m,q_m}+\Tilde{d}_m )) \right) \\
    &=1-\frac{\Tilde{c}_1+\Tilde{c}_2+...+\Tilde{c}_m}{k}.
\end{align*}
Note that line $3$ follows simply by the property that the limsup of a sequence of sets is unchanged by deleting a finite number of elements of the sequence, and hence deleting the first $k-1$ sets does not affect the limsup. On line $4$ we use the fact that the $H^{p,q}$ sequence is nested (see $(iv)$ of Lemma \ref{construction of Gpq sets}). The final line is achieved by applying $(iii)$ of Lemma \ref{construction of Gpq sets} to the $H^{p,q}$ sets which, by our choice of $\Tilde{c}_i$ and $\Tilde{d}_i$ do not intersect each other. As the value of $k$ was arbitrary, we deduce $\mu(\limsup_{i \rightarrow \infty }K_i^\delta)=1$. Therefore, $\mu(\limsup_{i \rightarrow \infty }B_i)=1$, this proves $(iii)$ of the lemma.
\end{proof}
We examine the consequences of Lemma $\ref{Ai lemma}$ and $\ref{Bi lemma}$ for our original problem. The inequalities in Lemmas \ref{Ai lemma} and \ref{Bi lemma} mean that to satisfy $(i)$ and $(ii)$ of Theorem \ref{the big theorem} it suffices to find appropriate values of $c_i, \Tilde{c_i}, \delta, p_i$ and $q_i$ that satisfy the following set of inequalities for $1 \leq r \leq m$.

\noindent For $r$ odd:
\begin{align}
\label{r odd problem}
     &c_1 \frac{p_1^r}{i_r q_1^r} + c_2 \frac{p_2^r}{i_r q_2^r} + ... + c_m  \frac{p_m^r}{i_r q_m^r}
    \geq
    \Tilde{c}_1 \frac{p_1^r}{i_r q_1^r} + \Tilde{c}_2 \frac{p_2^r}{i_r q_2^r} + ... + \Tilde{c}_k  \frac{p_m^r}{i_r q_m^r}+ \frac{\delta}{i_r}.
    \end{align}
    For $r$ even:
    \begin{align}
\label{r even number}
     &c_1 \frac{p_1^r}{i_r q_1^r} + c_2 \frac{p_2^r}{i_r q_2^r} + ... + c_m  \frac{p_m^r}{i_r q_m^r}
    \leq
    \Tilde{c}_1 \frac{p_1^r}{i_r q_1^r} + \Tilde{c}_2 \frac{p_2^r}{i_r q_2^r} + ... + \Tilde{c}_m  \frac{p_m^r}{i_r q_m^r}.
\end{align}

\subsubsection{Picking adequate constants to satisfy (\texorpdfstring{$\ref{r odd problem}$}{r odd problem}) and (\texorpdfstring{$\ref{r even number}$}{r even number}) }
    For $1 \leq i \leq m$, set $q_i=10^{(m+i-1)!}$ for $m \in \mathbb{N}$ as in the statement of Theorem \ref{the big theorem}. We will take $p_1=p_2=...=p_m=1$. Our choice of $p_i$ and $q_i$ ensures that the powers of $\frac{p_i}{q_i}$ will be of drastically different orders of magnitude. Now we see that we can rephrase $(\ref{r odd problem})$ and $(\ref{r even number})$ in terms of these choices of $p_i$ and $q_i$ and in a matrix form.

For two vectors $\mathbf{\underline{v}},\mathbf{\underline{w}} \in \mathbb{R}^m$ with $\mathbf{\underline{v}}= (v_1,...,v_m)$ and $\mathbf{\underline{w}}= (w_1,...,w_m)$ we will write $\mathbf{\underline{v}} \leq \mathbf{\underline{w}}$ if $v_i \leq w_i$ for all $1 \leq i \leq m$. First, we define the following $m \times m$ matrix and $m \times 1$ column vectors,
\[ M:=\begin{pmatrix} 
    -\dfrac{1}{10^{m!}} & - \dfrac{1}{10^{(m+1)!}} & \hdots & -\dfrac{1}{10^{(2m-1)!}} \\[1em]
    \dfrac{1}{10^{2\times m!}}& \dfrac{1}{10^{2 \times(m+1)!}} &  \hdots & \dfrac{1}{10^{2 \times (2m-1)!}} \\
    \vdots & \vdots & \ddots & \vdots \\
    \dfrac{(-1)^m}{10^{m \times m!}} & \dfrac{(-1)^m}{10^{m \times (m+1)!}} & \hdots & \dfrac{(-1)^m}{10^{m \times (2m-1)!}} 
    \end{pmatrix}, \underline{\text{\boldmath$\delta$}}:=
    \begin{pmatrix}
        \delta \\
        0 \\
        \delta \\
        0 \\
        \delta \\
        \vdots \\
        %%%%%%%not sure of the deltas
    \end{pmatrix},
\underline{\mathbf{c}}:=\begin{pmatrix}
        c_1 \\
        c_2\\
        c_3 \\
        \vdots \\
        c_m \\
    \end{pmatrix},
    \hspace{5pt}
    \text{and } \underline{\Tilde{\mathbf{c}}}:=
    \begin{pmatrix}
        \Tilde{c}_1 \\
        \Tilde{c}_2\\
        \Tilde{c}_3 \\
        \vdots \\
        \Tilde{c}_m \\
    \end{pmatrix}.
    \]

% \[
%     \begin{pmatrix} 
%     -\dfrac{p_1}{q_1} & -\dfrac{p_2}{q_2} & \hdots & -\dfrac{p_m}{q_m}\\[1em]
%     \dfrac{p_1^2}{q_1^2} & \dfrac{p_2^2}{q_2^2} & \hdots & \dfrac{p_m^2}{q_m^2} \\
%     \vdots & \vdots & \ddots & \vdots\\
%     \dfrac{(-p_1)^m}{q_1^m} & \dfrac{(-p_2)^m}{q_2^m} & \hdots & \dfrac{(-p_m)^m}{q_m^m} \\
%     \end{pmatrix}
%     \underline{\mathbf{c}}
%      \leq
%      \begin{pmatrix} 
%     -\dfrac{p_1}{q_1} & -\dfrac{p_2}{q_2} & \hdots & -\dfrac{p_m}{q_m}\\[1em]
%     \dfrac{p_1^2}{q_1^2} & \dfrac{p_2^2}{q_2^2} & \hdots & \dfrac{p_m^2}{q_m^2} \\
%     \vdots & \vdots & \ddots & \vdots\\
%     \dfrac{(-p_1)^m}{q_1^m} & \dfrac{(-p_2)^m}{q_2^m} & \hdots & \dfrac{(-p_m)^m}{q_m^m} \\
%     \end{pmatrix}
%      \underline{\Tilde{\mathbf{c}}}
%     -
%     \begin{pmatrix}
%         \delta \\
%         0 \\
%         \delta \\
%         0 \\
%         \delta \\
%         \vdots \\
%         %%%%%%%not sure of the deltas
%     \end{pmatrix},
% \]
% where:
% \[ \underline{\mathbf{c}}=\begin{pmatrix}
%         c_1 \\
%         c_2\\
%         c_3 \\
%         \vdots \\
%         c_m \\
%     \end{pmatrix},
%     \hspace{5pt}
%     \text{and } \underline{\Tilde{\mathbf{c}}}=
%     \begin{pmatrix}
%         \Tilde{c}_1 \\
%         \Tilde{c}_2\\
%         \Tilde{c}_3 \\
%         \vdots \\
%         \Tilde{c}_m \\
%     \end{pmatrix}
%     \] 

Noting that we can cancel the $i_r$ terms in (\ref{r odd problem}) and (\ref{r even number}), we can use $M, \underline{\Tilde{\mathbf{c}}}, \underline{\mathbf{c}},$ and $\underline{\text{\boldmath$\delta$}}$ to rephrase (\ref{r odd problem}) and (\ref{r even number}) as the vector inequality
\begin{align}\label{matrix inequality}
    M \underline{\mathbf{c}} \leq M  \underline{\mathbf{\Tilde{c}}}- \underline{\text{\boldmath$\delta$}}.
\end{align}
The remainder of the proof then breaks down into picking $\delta, c_i$ and $\Tilde{c}_i$ to ensure $\eqref{matrix inequality}$ is satisfied, which we show that can be done in the following lemma.

    \begin{lemma} \label{matrix lemma}
            There exists a positive real number $\delta$ and positive $c_i$, and $\Tilde{c}_i$ for $1 \leq i \leq m$ that satisfy \eqref{matrix inequality}.
        \end{lemma}
        % \[
%     \begin{pmatrix} 
%     -\dfrac{1}{10^{m!}}  & \hdots & -\dfrac{1}{10^{(2m-1)!}} \\[1em]
%     \dfrac{1}{10^{2\times m!}} &  \hdots & \dfrac{1}{10^{2 \times (2m-1)!}} \\
%     \vdots & \ddots & \vdots \\
%     \dfrac{(-1)^m}{10^{m \times m!}} & \hdots & \dfrac{(-1)^m}{10^{m \times (2m-1)!}} \\
%     \end{pmatrix}
%     \begin{pmatrix}
%         c_1 \\
%         c_2\\
%         c_3 \\
%         \vdots \\
%         c_m \\
%     \end{pmatrix}
%      \leq
%      \begin{pmatrix} 
%     -\dfrac{1}{10^{m!}}  & \hdots & -\dfrac{1}{10^{(2m-1)!}} \\[1em]
%     \dfrac{1}{10^{2\times m!}} &  \hdots & \dfrac{1}{10^{2 \times (2m-1)!}} \\
%     \vdots & \ddots & \vdots \\
%     \dfrac{(-1)^m}{10^{m \times m!}} & \hdots & \dfrac{(-1)^m}{10^{m \times (2m-1)!}} \\
%     \end{pmatrix}
%     \begin{pmatrix}
%         \Tilde{c}_1 \\
%         \Tilde{c}_2\\
%         \Tilde{c}_3 \\
%         \vdots \\
%         \Tilde{c}_m \\
%     \end{pmatrix}
%     -
%     \begin{pmatrix}
%         \delta \\
%         0 \\
%         \delta \\
%         \vdots \\
%         \delta \\
%         %%%%%%%not sure of the deltas
%     \end{pmatrix}
% \]

\begin{proof}
    
We will show the elements along the lower-left to upper-right diagonal to be the largest in their respective rows (when multiplied by their respective $c_i$). Specifically, for the row $a$ (where $1 \leq a \leq m)$, $M_{a,m-a+1}c_{m-a+1}$ should be significantly larger than any of the corresponding terms in this row. Therefore, for row $a$, our choice of $c_a$ and $\Tilde{c}_a$ will determine the direction of the inequality for row $a$.

To this end, for $2 \leq n \leq m$, we define $\gamma_n$ and $\alpha_n$ by
\begin{align*}
 \gamma_n&= \sum_{i=1}^{n-1}(m-\frac{1}{2} -(i-1)) ((
m+i)!-(m+i-1)!), \text{ and} \\\alpha_n &= \frac{1}{2}((m+n-1)!-(m+n-2)!).
\end{align*}
% \[\gamma_2= \frac{2m-1}{2} ((b+1)!-b!),\]
% \[\alpha_2=\frac{1}{2}((b+1)!-b!).\]
%\begin{align*}
 %   M_{a,m-a+1}c_{m-a+1} / M_{a,m-a+2}c_{m-a+2} \\ 
  %  10^{-a \times (b+m-a)}
   % 10^{\gamma_{m-a+1}} / 10^{-a \times (b+m-a+1)}10^{\gamma_{m-a+2}} \\
    %=10^{a}/ 10^{\gamma_{m-a+2}-\gamma_{m-a+1}} \\
    %=10^{a}/10^{(m-1/2 - (m-a+1-1))\times ((b+(m-a+1))! - (b+(m-a+1)-1)!} \\
    %=10^{a}/10^{.(a-1/2) \times ((b+m-a+1)!-(b+m-a)!)} \\
    %=10^{\frac{1}{2} ((b+m-a+1)!-(b+m-a)!)}
%\end{align*}

We set $\Tilde{c}_i$ to have the same magnitude as $c_i$, but multiply by 2 or $1/2$ to ensure that the matrix inequality holds. To this end, we have:
\begin{align*}
    c_1:=2 \text{ and } c_i:= \frac{3+(-1)^{i+1}}{2} \times 10^{\gamma_i} \text{ for } 2 \leq i \leq m ,\\
    \Tilde{c}_1:=1 \text{ and } \Tilde{c}_i:= \frac{3+(-1)^{i}}{2} \times 10^{\gamma_i} \text{ for } 2 \leq i \leq m .
\end{align*}
 This ensures $\Tilde{c}_i=2c_i$ if $i$ is even and $c_i=2\Tilde{c}_i$ if $i$ is odd. 
 We pick $\delta$ to be as follows
\[\delta:= \min_{\substack{1 \leq a \leq m \\ 1\leq r \leq m}} |M_{a,r} c_r| .\]
These choices of $c_i$ and $\Tilde{c}_i$ ensure that the lower left to upper right diagonal terms are larger than any of the other terms in their respective rows. To see this, we begin by showing that for each row, the diagonal term is larger than its adjacent terms in the row. That is,
\begin{align}\label{desired inequalities to satisfy}
\begin{split}
    |M_{m-a+1,a+1} c_{a+1}| &\leq 10^{-\alpha_a} |M_{m-a+1,a} c_{a}| \text{ for } 1 \leq a \leq m-1 ,\\
    |M_{m-a+1,a-1} c_{a-1}|&=10^{-\alpha_a} |M_{m-a+1,a} c_{a}| \text{ for } 2 \leq a \leq m .
\end{split}
\end{align}
We begin with $1 \leq a \leq m-1$ and show the first inequality
\begin{align*}
    \frac{|M_{m-a+1,a} c_{a}| }{| M_{m-a+1,a+1} c_{a+1}|}
    &=\frac{10^{-(m-a+1) \times (m+a-1)!} 10^{\gamma_a}}{ 10^{-(m-a+1) \times (m+a)!}10^{\gamma_{a+1}} }\\
    &= \frac{10^{(m-a+1)\times((m+a)!-(m+a-1)!)}}{10^{\gamma_{a+1}-\gamma_a}} \\
    &=\frac{10^{(m-a+1)\times((m+a)!-(m+a-1)!)} }{ 10^{(m-\frac{1}{2}-(a-1))\times((m+a)!-(m+a-1)!)}} \\
    &=10^{\frac{1}{2}((m+a)!-(m+a-1)!)} \\
    &=10^{\alpha_{a+1}} \geq 10^{\alpha_{a}} .
\end{align*}
For $2\leq a \leq m$ note that
\begin{align*}
    \frac{|M_{m-a+1,a} c_{a}| }{| M_{m-a+1,a-1} c_{a-1}|} &= \frac{10^{-(m-a+1) \times (m+a-1)!} 10^{\gamma_a} }{ 10^{-(m-a+1) \times (m+a-2)!}   10^{\gamma_{a-1}} } \\
    &= \frac{10^{\gamma_{a}-\gamma_{a-1}}}{10^{(m-a+1) \times ((m+a-1)!-(m+a-2)!)} }  \\
    &=\frac{10^{(m-\frac{1}{2}-(a-2))\times((m+a-1)!-(m+a-2)!)} }{ 10^{(m-a+1) \times ((m+a-1)!-(m+a-2)!)}} \\
    &= 10^{\frac{1}{2}((m+a-1)!-(m+a-2)!)} \\
    &=10^{\alpha_a} .
\end{align*}
The second part of \eqref{desired inequalities to satisfy} follows from this chain of equalities. Now we have shown that the diagonal terms are larger than their row adjacent terms, and we now show that this implies that they are larger than all terms in the row. We note that if $|M_{a,r}c_r|>|M_{a,s} c_s|$, where $1 \leq r<s\leq m$ and $1 \leq a \leq m-1$, since the column $s$ of the matrix $M$ shrinks faster than the column $r$, we have $|M_{a+1,r}c_r| > |M_{a+1,s} c_s|$. We precisely justify this as follows:
\begin{align*}
    |M_{a+1,r}c_r|&=10^{-(b+r-1)!} |M_{a,r} \times c_r | \\
    &>10^{-(b+s-1)!} |M_{a,r} \times c_r | \\
     &>10^{-(b+s-1)!} |M_{a,s}  \times c_s |=|M_{a+1,s}c_s|.
\end{align*}
Also, for $2 \leq a \leq m$ if $|M_{a,r}c_r|<|M_{a,s} c_s|$ (again for $1 \leq r<s \leq m$) by similar logic $|M_{a-1,r}c_r|<|M_{a-1,s} c_s|$. This ensures that the terms in the row, excluding the non-diagonal term, can be bounded by the terms adjacent to the diagonal term. Along with \eqref{desired inequalities to satisfy}, where we see the diagonal terms bound their adjacent terms, we obtain
\begin{align*}
    \sum_{i \in \{1,2,...,m\}^{}
    \backslash a} |M_{m-a+1,i}{c_i}| +\delta &\leq m\max \{M_{m-a+1,a+1}c_{a+1},M_{m-a+1,a-1}c_{a-1} \} \\
    &\leq m 10^{-\alpha_a}  |M_{m-a+1,a}c_a| \\ &\leq m 10^{-\frac{1}{2}(m+a-1)} |M_{m-a+1,a}c_a|
     \\&< \frac{1}{4}|M_{m-a+1,a}c_a|.
\end{align*}

Thus, $M_{m-a+1,a}c_a$ dominates the sum of the rest of the row. The same proof ensures the right hand side of \eqref{matrix inequality} can be bounded in the same way and we see that $M_{m-a+1,a}\Tilde{c_a}$ dominates the row $a$ of the right hand side of \eqref{matrix inequality}.
Combining all this together, we see that the row $a$ (for $1 \leq a \leq m$) of the left hand side of $\eqref{matrix inequality}$ can be bounded above and below by
\begin{align*}
    \frac{3}{4} |M_{m-a+1,a}c_a| <(M \underline{\mathbf{c}})_a < \frac{5}{4} |M_{m-a+1,a}c_a|.
\end{align*}
Similarly row $a$ of the right hand side of $\eqref{matrix inequality}$ can be bounded by
\begin{align*}
    \frac{3}{4} |M_{m-a+1,a}\Tilde{c}_a| <(M \underline{\mathbf{\Tilde{c}}}+\underline{\text{\boldmath$\delta$}})_a < \frac{5}{4} |M_{m-a+1,a} \Tilde{c}_a|.
\end{align*}
Therefore, the inequality in this row is controlled by our choice of $c_a$ and $\Tilde{c_a}$ from earlier. This ensures the matrix inequality \eqref{matrix inequality} holds true for the prescribed $\delta$ and $c_i, \Tilde{c}_i$ for $1 \leq i \leq m$ as desired. \qedhere

\end{proof}

\begin{proof}[Proof of Theorem \ref{the big theorem}]
    The sequences $(A_i)_{i \in \mathbb{N}}$ and $(B_i)_{i \in \mathbb{N}}$ defined in this chapter are the sequences satisfying the conditions of the theorem. By Lemma \ref{matrix lemma} we have shown that there exist constants ensuring the matrix inequality, and thus expressions (\ref{r odd problem}) and (\ref{r even number}) are satisfied. We can shrink all of the $c_i$ and $\Tilde{c_i}$ by equal factors if necessary to ensure that the sets lie entirely within $[0,1]$ without affecting the desired inequalities of (\ref{r odd problem}) and (\ref{r even number}). This proves parts $(i)$ and $(ii)$ of Theorem \ref{the big theorem}. We proved that $\mu(\limsup_{i \rightarrow \infty} A_i)=0$ and $\mu(\limsup_{i \rightarrow \infty} B_i)=1$ by properties $(iii)$ and $(iv)$ in Lemmas $\ref{Ai lemma}$ and \ref{Bi lemma} so have demonstrated the existence of two sequences satisfying the conditions of Theorem \ref{the big theorem}.
\end{proof}

\begin{remark}
   We see that Theorem \ref{the big theorem} is not specific to the 1 dimensional Lebesgue measure. Analogous examples can be constructed in higher dimensions by taking the sets that satisfy Theorem $\ref{the big theorem}$ in dimension $1$ and appending `$\times [0,1]^{n-1}$' at the end of the sets, and we have ourselves an example in $\mathbb{R}^n$.

   \end{remark}

% Despite this result holding, it would still be interesting to find some strengthened forms of the Borel Cantelli lemma and perhaps formulate some new forms dependent on more topological properties similar to that done by \cite{Velanipaper}. This would ultimately go some way in providing new tools to develop theory in areas such as Graph Theory, Diophantine approximation, Dynamical Systems and more.

\bibliographystyle{alpha}
\bibliography{sample}

\section*{Appendix}

Here we include figures of the first 15 sets of $(A_i)_{i \in \mathbb{N}}$ and $(B_i)_{i \in \mathbb{N}}$ of Theorem \ref{the big theorem} (in the case $c=1$ and $m=2$) to aid the readers understanding of the construction and see that the measures restrictions on the sets and intersections is satisfied. Figure \ref{leb construction equal Bi} visibly show how the sequence $(B_i)_{i \in \mathbb{N}}$ `spreads out' across the entire interval ensuring that the limsup will be of full measure. Alternatively, Figure \ref{leb construction equal Ai} shows the sequence $(A_i)_{i \in \mathbb{N}}$ neglects entire intervals.

\begin{figure*}[h!]
\renewcommand{\thefigure}{A1}
    \caption{Example of the first 15 terms of the sequence $(A_i)_{i \in \mathbb{N}}$ in Theorem \ref{the big theorem} for $m=2$ and $c=1$}
    \label{leb construction equal Ai}
\begin{center}
    \begin{tikzpicture}[yscale=0.5]
    \draw[very thick] (0,5) -- (12,5);
    \draw (0,4) -- (12/8*4,4);
    \draw (0,3) --(2*12/8,3);
    \draw (4*12/8,3) --(6*12/8,3);
    \draw (2*12/8,2) -- (4*12/8,2);
    \draw (4*12/8,2) -- (6*12/8,2);

    \draw (0,1) -- (1.5,1);
\draw (0,0) -- (0.75,0);
\draw (1.5,0) -- (2.25,0);
\draw (0.75,-1) -- (1.5,-1);
\draw (1.5,-1) -- (2.25,-1);
\draw (3,1) -- (4.5,1);
\draw (3,0) -- (3.75,0);
\draw (4.5,0) -- (5.25,0);
\draw (3.75,-1) -- (4.5,-1);
\draw (4.5,-1) -- (5.25,-1);

\draw (0,-2) -- (1.5,-2);
\draw (0,-3) -- (0.75,-3);
\draw (1.5,-3) -- (2.25,-3);
\draw (0.75,-4) -- (1.5,-4);
\draw (1.5,-4) -- (2.25,-4);
\draw (6,-2) -- (7.5,-2);
\draw (6,-3) -- (6.75,-3);
\draw (7.5,-3) -- (8.25,-3);
\draw (6.75,-4) -- (7.5,-4);
\draw (7.5,-4) -- (8.25,-4);

\draw (6,-5) -- (7.5,-5);
\draw (6,-6) -- (6.75,-6);
\draw (7.5,-6) -- (8.25,-6);
\draw (6.75,-7) -- (7.5,-7);
\draw (7.5,-7) -- (8.25,-7);
\draw (3,-5) -- (4.5,-5);
\draw (3,-6) -- (3.75,-6);
\draw (4.5,-6) -- (5.25,-6);
\draw (3.75,-7) -- (4.5,-7);
\draw (4.5,-7) -- (5.25,-7);

\draw (0,-8) -- (0.375,-8);
\draw (0,-9) -- (0.1875,-9);
\draw (0.375,-9) -- (0.5625,-9);
\draw (0.1875,-10) -- (0.375,-10);
\draw (0.375,-10) -- (0.5625,-10);
\draw (0.75,-8) -- (1.125,-8);
\draw (0.75,-9) -- (0.9375,-9);
\draw (1.125,-9) -- (1.3125,-9);
\draw (0.9375,-10) -- (1.125,-10);
\draw (1.125,-10) -- (1.3125,-10);

\draw (3,-8) -- (3.375,-8);
\draw (3,-9) -- (3.1875,-9);
\draw (3.375,-9) -- (3.5625,-9);
\draw (3.1875,-10) -- (3.375,-10);
\draw (3.375,-10) -- (3.5625,-10);
\draw (3.75,-8) -- (4.125,-8);
\draw (3.75,-9) -- (3.9375,-9);
\draw (4.125,-9) -- (4.3125,-9);
\draw (3.9375,-10) -- (4.125,-10);
\draw (4.125,-10) -- (4.3125,-10);

    \node at (-1/2,5) {[0,1]};
    \node at (-1/2,4) {$A_1$};
    \node at (-1/2,3) {$A_2$};
    \node at (-1/2,2) {$A_3$};
    \node at (-1/2,1) {$A_4$};
    \node at (-1/2,0) {$A_5$};
    \node at (-1/2,-1) {$A_6$};
    \node at (-1/2,-2) {$A_7$};
    \node at (-1/2,-3) {$A_8$};
    \node at (-1/2,-4) {$A_9$};
    \node at (-1/2,-5) {$A_{10}$};
    \node at (-1/2,-6) {$A_{11}$};
    \node at (-1/2,-7) {$A_{12}$};
    \node at (-1/2,-8) {$A_{13}$};
    \node at (-1/2,-9) {$A_{14}$};
    \node at (-1/2,-10) {$A_{15}$};
    
    \node at (0,1/2-11.5) {0};
    \node at (12/8,1/2-11.5) {$\frac{1}{8}$};
    \node at (24/8,1/2-11.5) {$\frac{2}{8}$};
    \node at (3*12/8,1/2-11.5) {$\frac{3}{8}$};
    \node at (4*12/8,1/2-11.5) {$\frac{4}{8}$};
    \node at (5*12/8,1/2-11.5) {$\frac{5}{8}$};
    \node at (6*12/8,1/2-11.5) {$\frac{6}{8}$};
    \node at (7*12/8,1/2-11.5) {$\frac{7}{8}$};
    \node at (12,1/2-11.5) {1};
    \end{tikzpicture}
\end{center}
\end{figure*}

\begin{figure}[h!]
\renewcommand{\thefigure}{A2}
    \caption{Example of the first 15 terms of the sequence $(B_i)_{i \in \mathbb{N}}$ in Theorem \ref{the big theorem} for $m=2$ and $c=1$}
    \label{leb construction equal Bi}
\begin{center}
    \begin{tikzpicture}[yscale=0.5]
    \draw[very thick] (0,5) -- (12,5);
    \draw (0,4) -- (6,4);
\draw (0,3) -- (3,3);
\draw (6,3) -- (9,3);
\draw (0,2) -- (3,2);
\draw (9,2) -- (12,2);

\draw (0,1) -- (1.5,1);
\draw (0,0) -- (0.75,0);
\draw (1.5,0) -- (2.25,0);
\draw (0,-1) -- (0.75,-1);
\draw (2.25,-1) -- (3,-1);
\draw (3,1) -- (4.5,1);
\draw (3,0) -- (3.75,0);
\draw (4.5,0) -- (5.25,0);
\draw (3,-1) -- (3.75,-1);
\draw (5.25,-1) -- (6,-1);

\draw (0,-2) -- (1.5,-2);
\draw (0,-3) -- (0.75,-3);
\draw (1.5,-3) -- (2.25,-3);
\draw (0,-4) -- (0.75,-4);
\draw (2.25,-4) -- (3,-4);
\draw (6,-2) -- (7.5,-2);
\draw (6,-3) -- (6.75,-3);
\draw (7.5,-3) -- (8.25,-3);
\draw (6,-4) -- (6.75,-4);
\draw (8.25,-4) -- (9,-4);

\draw (0,-5) -- (1.5,-5);
\draw (0,-6) -- (0.75,-6);
\draw (1.5,-6) -- (2.25,-6);
\draw (0,-7) -- (0.75,-7);
\draw (2.25,-7) -- (3,-7);
\draw (9,-5) -- (10.5,-5);
\draw (9,-6) -- (9.75,-6);
\draw (10.5,-6) -- (11.25,-6);
\draw (9,-7) -- (9.75,-7);
\draw (11.25,-7) -- (12,-7);

\draw (0,-8) -- (0.375,-8);
\draw (0,-9) -- (0.1875,-9);
\draw (0.375,-9) -- (0.5625,-9);
\draw (0,-10) -- (0.1875,-10);
\draw (0.5625,-10) -- (0.75,-10);
\draw (0.75,-8) -- (1.125,-8);
\draw (0.75,-9) -- (0.9375,-9);
\draw (1.125,-9) -- (1.3125,-9);
\draw (0.75,-10) -- (0.9375,-10);
\draw (1.3125,-10) -- (1.5,-10);

\draw (3,-8) -- (3.375,-8);
\draw (3,-9) -- (3.1875,-9);
\draw (3.375,-9) -- (3.5625,-9);
\draw (3,-10) -- (3.1875,-10);
\draw (3.5625,-10) -- (3.75,-10);
\draw (3.75,-8) -- (4.125,-8);
\draw (3.75,-9) -- (3.9375,-9);
\draw (4.125,-9) -- (4.3125,-9);
\draw (3.75,-10) -- (3.9375,-10);
\draw (4.3125,-10) -- (4.5,-10);

    \node at (-1/2,5) {[0,1]};
    \node at (-1/2,4) {$B_1$};
    \node at (-1/2,3) {$B_2$};
    \node at (-1/2,2) {$B_3$};
    \node at (-1/2,1) {$B_4$};
    \node at (-1/2,0) {$B_5$};
    \node at (-1/2,-1) {$B_6$};
    \node at (-1/2,-2) {$B_7$};
    \node at (-1/2,-3) {$B_8$};
    \node at (-1/2,-4) {$B_9$};
    \node at (-1/2,-5) {$B_{10}$};
    \node at (-1/2,-6) {$B_{11}$};
    \node at (-1/2,-7) {$B_{12}$};
    \node at (-1/2,-8) {$B_{13}$};
    \node at (-1/2,-9) {$B_{14}$};
    \node at (-1/2,-10) {$B_{15}$};
    
    \node at (0,1/2-11.5) {0};
    \node at (12/8,1/2-11.5) {$\frac{1}{8}$};
    \node at (24/8,1/2-11.5) {$\frac{2}{8}$};
    \node at (3*12/8,1/2-11.5) {$\frac{3}{8}$};
    \node at (4*12/8,1/2-11.5) {$\frac{4}{8}$};
    \node at (5*12/8,1/2-11.5) {$\frac{5}{8}$};
    \node at (6*12/8,1/2-11.5) {$\frac{6}{8}$};
    \node at (7*12/8,1/2-11.5) {$\frac{7}{8}$};
    \node at (12,1/2-11.5) {1};
    \end{tikzpicture}
\end{center}
\end{figure}

\end{document}